\documentclass[11pt]{amsart}

\usepackage[T1]{fontenc}
\usepackage[utf8]{inputenc}
\usepackage{lmodern}
\usepackage[margin=1.05in]{geometry}
\usepackage{microtype}
\usepackage{amsmath,amssymb,amsthm,mathtools}
\usepackage{booktabs}
\usepackage{xcolor}
\usepackage[hidelinks]{hyperref}
\usepackage[nameinlink,noabbrev]{cleveref}

\numberwithin{equation}{section}
\newtheorem{theorem}{Theorem}[section]
\newtheorem{proposition}[theorem]{Proposition}
\newtheorem{corollary}[theorem]{Corollary}

\theoremstyle{definition}
\newtheorem{definition}[theorem]{Definition}

\theoremstyle{remark}
\newtheorem{remark}[theorem]{Remark}
\newtheorem{example}[theorem]{Example}

\newcommand{\norm}[1]{\lVert#1\rVert}
\newcommand{\lamin}{\lambda_{\min}}

\newcommand{\R}{\mathbb R}
\newcommand{\Sph}{\mathbb S}

\newcommand{\cB}{\mathcal B}
\newcommand{\cH}{\mathcal H}
\newcommand{\cG}{\mathcal G}

\newcommand{\lmin}{\lambda_{\min}}
\newcommand{\lmax}{\lambda_{\max}}

\begin{document}
\title[cubic and quartic landscapes]{Benign cubic and quartic forms on  the Sphere}
\thanks{This research is supported by the AI Interdisciplinary Institute ANITI funding through the ANITI AI
Cluster program under the Grant agreement ANR-23-IACL-0002, and also by the Marie-Sklodovska-Curie European
doctoral network TENORS, grant 101120296. This research is also part of the programme DesCartes and is supported by
the National Research Foundation, Prime Minister's Office, Singapore under its Campus for Research Excellence and
Technological Enterprise (CREATE) programme.}
\author{Jean B. Lasserre}
\address{LAAS-CNRS and Toulouse School of Economics (TSE)\\
University of Toulouse\\
LAAS, 7 avenue du Colonel Roche, BP 54200\\
31031 Toulouse Cedex 4, France}
\email{lasserre@laas.fr}

\date{}

\begin{abstract}
One considers the NP-hard problem of minimizing real cubic and quartic forms on the Euclidean unit
sphere. One first extends the Hessian spectral characterization of local minima  by the author. In particular one provides a gap--depth characterization of negative global and spurious minima,
and proves that the minimum value
of an odd degree-$d$ form is recovered (up to the factor $d(d-1)$) 
by minimizing the smallest eigenvalue of its ambient Hessian 
over the sphere (and equivalently the unit ball).
Next, one exhibits several classes of \emph{benign} families of cubic and quartic forms
(i.e. with no spurious local minima on the sphere)
and prove that some are robust to small coefficient perturbation
(i.e. the resulting form remains benign even if it leaves the initial family).
Therefore the class of benign cubic and quartic forms is full-dimensional
(in the ambient space of coefficients), for $n\geq2$.
Also for every odd degree $d$,
the class of benign degree-$d$ forms increases with the dimension (however it does not
imply that such forms become more frequent when the dimension increases).
\end{abstract}
\maketitle

\section{Introduction}

Minimizing a cubic or quartic form on the Euclidean sphere 
$\mathbb{S}^{n-1}$
($n\geq2$) is one of the simplest optimization problem (to describe). Yet it  models
several well-known NP-hard problems and therefore it is worth investigating its 
characteristics. In particular, an important  topic is to
characterize the class of \emph{benign} cubic and/or quartic forms
(i.e. with no \emph{spurious minima} on $\mathbb{S}^{n-1}$), and if possible provide explicit
families in that 
class. Indeed if a form is identified as benign then any local minimization algorithm
that converges to a local minimum will obtain the highly desired global minimum (and a 
global minimizer as well) with no need to resort to a more sophisticated (and computationnally
more expensive) algorithm of global optimization.

More precisely, for $d\ge1$, let $\cH_{n,d}\subset\R[x]$ be
be the space of degree-$d$ forms, hence  with $\dim\cH_{n,d}=\binom{n+d-1}{d}$.
In particular,
cubics and quartics belong to $\cH_{n,3}$ and $\cH_{n,4}$, respectively.
With the Euclidean inner product one may and will identify  $\cH_{n,d}$ with the space of
symmetric order-$d$ tensors:
\begin{equation}\label{eq:forms-tensors-isomorphism}
 \cH_{n,d}\simeq\operatorname{Sym}^d(\R^n),
 \qquad
 f(x)=\langle T_f,x^{\otimes d}\rangle.
\end{equation}
Thus while $f$ and $T_f$ describe equivalent objects, they emphasize different structures.  
Next, with $n\geq 2$, consider the homogeneous optimization problem

\begin{equation}\label{eq:problem}
 f_\star\,=\min_{x\in\mathbb{S}^{n-1}} f(x)\,=\min\{\,f(x):\|x\|^2=1\,\}\,.
\end{equation}
In particular for cubics,
\begin{equation}\label{eq:odd}
 \min_{\|x\|=1}f(x)
 =-\max_{\|x\|=1}f(x)
 =-\max_{\|x\|=1}|f(x)|\,,
\end{equation}
and observe that the last quantity is the negative spectral norm of the symmetric tensor $T_f$.
Hillar and Lim \cite[Theorem~10.2]{HillarLim2013} proved that computing the spectral norm and a best symmetric
rank-one approximation of a rational symmetric third-order tensor are NP-hard.
Thus \eqref{eq:problem} is itself
an NP-hard cubic optimization problem, not merely a problem related to a hard
tensor task.

For more details on complexity results and optimization algorithms/relaxations  for optimization on the unit sphere
the interested reader is referred to \cite[\S 3-4]{deKlerk2008Survey} and \cite{Faybusovich2004,LLWY2014,Nie2023,NieWang2014}

A local minimizer of $f|_{\Sph^{n-1}}$ is called \emph{spurious} when its value
is strictly larger than $f_\star$ (then called a spurious minimum).  
In the absence of spurious minima, the form is said to have a \emph{benign landscape}.
Trivially, quadratic forms have no spurious minimum 
and hence have a benign landscape.  
Therefore degree-$3$ is the first homogeneous case in which tensor
hardness and a nontrivial landscape coexist.  Degree-$4$ is at least as rich as
e.g., structured biquadratic forms encode graph and copositivity problems.

In \cite{LasserreOL2022} the author identified a remarkably 
economical test for \emph{local} optima. For instance,
at a negative critical point $x\in\mathbb{S}^{n-1}$ of
a cubic, the Second-Order Necessary Condition (SONC) has a compact statement in terms
of  the two smallest eigenvalues $(\lambda_1,\lambda_2)$ of its \emph{ambient} Hessian $\nabla^2f(x)$; moreover
Huang \cite{Huang2023} subsequently proved the
\emph{generic} implication from first- and second-order necessary conditions to the
second-order sufficient condition (SOSC), resolving the genericity issue in 
\cite{LasserreOL2022}.   Finally, in \cite{LasserreSIAM2022} the author considered polynomial optimization on $\mathbb{S}^{n-1}$
when the objective depends on few (say $m$) linear forms only. Then the problem
reduces from the unit sphere in $n$ variables to the unit ball in $m$ variables,
and several structured classes have no spurious minima.

In structured matrix and tensor learning problems, Ge et al. \cite{GeJinZheng2017}
have proven benign landscapes for specific problems with low-rank formulations 
and under structural hypotheses.

\subsection{Contribution}
\label{subsec:contributions}

 For the paper to be self-contained, one first re-derives the spectral 
 characterization in \cite{LasserreOL2022}  of spherical SONC points  (hence generically, of local minimizers,
 thanks to Huang's result \cite{Huang2023} on genericity). 
Then our main contribution is as follows:
\begin{enumerate}
 \item We sharpen the negative cubic test of \cite{LasserreOL2022} 
 by showing that, at a negative
 spherical critical point, the single inequality
 \(\lambda_1+\lambda_2\ge9f\) is equivalent to SONC, and  establish 
 \begin{equation}
 \label{eq:intro}
 6\,f_\star\,=\,
  \min_{\|y\|=1}\lambda_1(\nabla^2f(y))\,=\,\min_{\|y\|=1}\max_{\lambda\in\R}\, \{\lambda:\nabla^2f(y)\,
  \succeq\lambda\mathrm{I}\,\}\,.
 \end{equation}
That is, \eqref{eq:intro} provides  a ``min-max" formulation and a  spectral identity for the global minimum
which complements the local gap test of \cite{LasserreOL2022} and
also yields an exact generic characterization of negative spurious minima. In fact \eqref{eq:intro} extends 
to odd-degree forms:
\begin{equation}
 \label{eq:intro-odd}
 d(d-1)\,f_\star\,=\,
  \min_{\|y\|=1}\lambda_1(\nabla^2f(y))\,=\,\min_{\|y\|=1}\max_{\lambda\in\R}\, \{\lambda: \nabla^2f(y)\,
  \succeq\lambda\mathrm{I}\,\}\,.
 \end{equation}
By analogy with the classical characterization of the smallest eigenvalue of a symmetric matrix,
\eqref{eq:intro-odd} 
provides an odd-degree exact Hessian spectral identity, analogue of the Rayleigh--Ritz formula with
a sharp distinction between the odd and even degree cases.  
(The cubic case is also distinguished from arbitrary odd forms by the fact that the Hessian is a \emph{linear} pencil.)
It complements classical $Z$-eigenvalue results from alternative tensor formulations in
\cite{KoldaMayo2011,Lim2005,Qi2005}.
 \item We also show that for all odd degrees, benignity is preserved under a simple ``zero-extension",
 and so  benign class $\mathcal G_n$ strictly increases with the dimension, 
 i.e. $\mathcal G_{n}\subsetneq\mathcal G_{n+1}$ for all $n$. However one
 does not establish that benign forms become more frequent as dimension grows. The coefficient-space dimensions change, and no probability measure on those spaces has been specified.
 \item We identify and analyze several classes of nontrivial parametrized benign cubic forms:
 Namely, the height cubic, eikonal
  cubics, what we call an \emph{axial} family with an exact eigenvalue criterion, and inspired 
 by the latter, a dimension-independent class of genuinely non-commuting symmetric pencils.
  \item We prove ambient robustness results.  Namely, we say that a form is  \emph{robustly benign} 
 if it remains benign 
 after any sufficiently small perturbation in the full coefficient space; that is,  the form is robust against
 such perturbations of its coefficients.
 
 We show that some of the above benign families are indeed robust to perturbations
 under some conditions  on the family parameters. Interestingly, 
 even though such perturbations  generally destroy the algebraic structure defining the original benign family,
 the resulting form is still benign.  
 This is a distinguishing feature when compared with exact benignity.
 
 The above benign families examples are illustrative. For each value of the parameter
 the global minimum of the form is easy to obtain. However, what is interesting
 is that any (sufficiently small) perturbation yields a form not in the family in general,
 and for which the global minimum is not easy to obtain.
 
  \item For quartics, we derive a complete semidefinite criterion for the
 family of forms
 \[
  f_{\rho,C}(t,y)
  =\rho(t^2+\|y\|^2)^2+t^2y^{\mathsf T}Cy,
 \]
 to be benign. In addition, such quartics are robustly benign relative to their parameter space$(\rho,C)$
 (and not necessarily to arbitrary small perturbation of their coefficients)
 

\end{enumerate}

\section{Sphere, ball, and computational hardness}
In the remainder of the paper one assumes that $n\geq2$. For the paper to be  self-contained we first review some spectral characterizations described 
in \cite{LasserreOL2022}.
At $x\in\Sph^{n-1}$, let $\Pi_x=I-xx^{\mathsf T}$ denote the projection on $x^\perp$.
Next, with $T_x\mathbb{S}^{n-1}$ being the tangent space
to $\mathbb{S}^{n-1}$ at $x\in\Sph^{n-1}$, $T_x\Sph^{n-1}=x^\perp$.
We first consider a degree-$d$ form $f$.

\begin{definition}[FONC, SONC, and SOSC]
A point $x\in\Sph^{n-1}$ satisfies the spherical first-order necessary
condition (FONC) when
\begin{equation}\label{eq:fonc-d}
 \nabla f(x)=df(x)x.
\end{equation}
At a FONC point, the intrinsic Hessian $\nabla^2_Sf:\,x^\perp \to\,x^\perp$
is the ambient Hessian $\nabla^2f(x)$ minus 
$df(x)\,I$, and then restricted to $x^\perp$, i.e.,
\begin{equation}\label{eq:intrinsic}
 \nabla^2_S f(x)
 =\left.\Pi_x\bigl(\nabla^2f(x)-df(x)I\bigr)\Pi_x\right|_{x^\perp}\,.
\end{equation}
With the ambient Hessian, the second-order necessary condition (SONC) reads 
\[u^T(\nabla^2f(x)-d\,f(x)\,I)\,u\,\geq\,0\,,\quad \forall u\in\,x^\perp\,,\]
while it reads 
$\nabla^2_S f(x)\succeq0$ with intrinsic Hessian. With the strict inequality $\succ0$,  it is the second-order sufficient condition (SOSC).
\end{definition}

If $f$ is a form then so are its partial derivatives. Then by Euler's identity,
\begin{equation}\label{eq:euler}
 x^{\mathsf T}\nabla f(x)=df(x)\,;
 \qquad \nabla^2f(x)x=(d-1)\nabla f(x).
\end{equation}
Hence, every FONC point $x\in\mathbb{S}^{n-1}$ is an ambient-Hessian eigenvector with associated \emph{radial}
eigenvalue $d(d-1)f(x)$, i.e., $\nabla^2f(x)\,x=d(d-1)f(x)\,x$.

\begin{proposition}[Negative-value criterion \cite{LasserreOL2022}]\label{prop:lasserre}
Let $d>2$, let $x\in\Sph^{n-1}$ satisfy FONC, and assume $f(x)<0$.  If
$\lambda_1\le\cdots\le\lambda_n$ are the eigenvalues of $\nabla^2f(x)$, then
SONC is equivalent to
\begin{equation}\label{eq:general-gap}
 \lambda_1=d(d-1)f(x)\,;
 \qquad \lambda_2\ge df(x)\,,
\end{equation}
and with the strict second inequality, \eqref{eq:general-gap} is equivalent to SOSC.
\end{proposition}

\begin{proof}
The \emph{only if part}. 
SONC says that every
tangent eigenvalue is at least $df(x)$.  Since
$d(d-1)f(x)<df(x)$ for $f(x)<0$ and $d>2$, the radial eigenvalue $d(d-1)f(x)$ is the unique
lowest eigenvalue under the strict condition and is a lowest eigenvalue under
SONC.  For the \emph{if part}, since $\lambda_1=d(d-1)f(x)$, $x$ is the eigenvector associated with $\lambda_1$,
and therefore suppose that
\[d\,f(x)\leq\,\lambda_2=\min_{u\in x^\perp;\Vert u\Vert=1}u^T\nabla^2f(x)u\,,\]
with the $\min$ attained at $u_*$. Then from $\Pi_xu_*=(I-xx^T)\,u_*=u_*$, and 
\[\Pi_x\nabla^2f(x)\,=\,\nabla^2f(x)-xx^T\nabla^2f(x)=\nabla^2f(x)-d(d-1)f(x)\,xx^T\,,\]
\begin{eqnarray*}
u^T_*\Pi_x\left(\nabla^2f(x)-d\,f(x)\,I\right)\Pi_xu_*&=&u^T_*\Pi_x\left(\nabla^2f(x)-d\,f(x)\,I\right)\,u_*\\
&=&u^T_*\Pi_x(\nabla^2f(x)\,u_*-d\,f(x)\,u_*)\\
&=&u^T_*\nabla^2f(x)u_*-d\,f(x)\,\geq\,0\,,
\end{eqnarray*}
which proves that $0\leq \lambda_2-d\,f(x)=\lambda_{\min}(\nabla^2_S f(x))$, i.e.,
$\nabla^2_Sf(x))\succeq0$, and so  $x$ is a SONC point.
 \end{proof}

For a cubic, introduce the normalized lowest-eigenvalue gap
\begin{equation}\label{eq:G}
 G(x)=\frac{\lambda_2(\nabla^2f(x))-\lambda_1(\nabla^2f(x))}
 {|\lambda_1(\nabla^2f(x))|}.
\end{equation}

\begin{corollary}[The cubic threshold]\label{cor:cubic-gap}
At a negative FONC point $x\in\mathbb{S}^{n-1}$ of a cubic form,
\begin{equation}\label{eq:cubic-gap}
 \mathrm{SONC}\quad\Longleftrightarrow\quad
 [\,\lambda_1=6f(x),\quad \lambda_2\ge3f(x)\,]
 \quad\Longleftrightarrow\quad [\,\lambda_1=6f(x)\mbox{ and } G(x)\ge\frac12\,],
\end{equation}
and  \eqref{eq:cubic-gap}  with the strict threshold $G(x)>\frac12$, 
guarantees SOSC.
\end{corollary}

Write $\cB_n=\{x\in\R^n:\|x\|\le1\}$ for the Euclidean ball.  For a nonzero odd-degree form, sphere
and ball minimization are exactly equivalent. 

\begin{proposition}[Sphere--ball equivalence]
\label{prop:ball}
Let $f$ be a nonzero odd-degree form.  Then
\begin{equation}\label{eq:ball}
 \min_{x\in\cB_n}f(x)=\min_{x\in\Sph^{n-1}}f(x)\,<\,0,
\end{equation}
and the global minimizers in the ball lie on the sphere.  Moreover, a negative
local minimizer on the sphere is a local minimizer relative to the ball, and
conversely every negative local minimizer in the ball lies on the sphere and is
a spherical local minimizer.
\end{proposition}

\begin{proof}
Let $f$ be a degree-$d$ form with $d=2s+1$, and $m=\min_{\Sph^{n-1}}f<0$.  Every nonzero $x\in\cB_n$ has the form $x=ru$
with $u\in\Sph^{n-1}$ and $0<r\le1$.  Homogeneity gives
$f(ru)=r^df(u)$.  If $f(u)\ge0$, then $f(ru)\ge0>m$; if $f(u)<0$, then
$r^df(u)\ge f(u)\ge m$.  This proves the equality and the statement about
global minimizers.  A negative interior point (i.e. with $f(x)<0$ and $\Vert x\Vert <1$) cannot be locally minimizing,
because a small outward radial move decreases its value.  At a
negative boundary point, moving inward radially increases the value, so the
local statements follow.  The restriction to negative local minima is
essential: a form  such as $x_1^{2s}x_2$ can have zero-valued, nonisolated local
minima in the interior region where $x_2>0$ and $x_1=0$.
\end{proof}

The ball $\mathcal{B}_n$ is convex, but the objective is not, and therefore convexity of the domain
does not remove NP-hardness or spurious local minima.  On the other
hand, the ball formulation can be useful for relaxations, for minimization algorithms and for 
dimension reduction as already exploited in  \cite{LasserreSIAM2022}.

\subsection{A concrete NP-hard cubic tensor problem}
\label{subsec:cubic-np-hardness}

The hardness already appears in the exact class studied in this paper.  For a
rational symmetric tensor $T\in\operatorname{Sym}^3(\mathbb Q^n)$, put
\begin{equation}\label{eq:symmetric-tensor-cubic}
 f_T(x)=\langle T,x^{\otimes3}\rangle.
\end{equation}
By Banach's symmetric-tensor theorem,
\begin{equation}\label{eq:cubic-spectral-norm}
 \|T\|_\sigma
 =\max_{\|x\|=1}|f_T(x)|
 =-\min_{\|x\|=1}f_T(x)
 =-\min_{\|x\|\le1}f_T(x).
\end{equation}
Hillar and Lim \cite[Theorems~6.9 and~10.2]{HillarLim2013} prove that computing this quantity is NP-hard, remains NP-hard
under the symmetry restriction, and admits no FPTAS unless ${\rm P}={\rm NP}$.
Equivalently, determining the
best symmetric rank-one approximation
\begin{equation}\label{eq:best-symmetric-rank-one}
 \min_{\lambda\ge0,\,\|u\|=1}
 \|T-\lambda u^{\otimes3}\|_F
\end{equation}
is NP-hard, with the optimal $\lambda$ equal to $\|T\|_\sigma$
\cite[Theorem~10.1]{HillarLim2013}.

Hence an exact solution
of the cubic sphere problem recovers the NP-hard stability number; see also 
\cite[\S 3.3]{deKlerk2008Survey}.  Oddness
turns the maximum in \eqref{eq:cubic-spectral-norm} into the corresponding
minimum, and Proposition \ref{prop:ball} shows that replacing the sphere by the convex
unit ball does not change the optimum.  Consequently,
\begin{equation}\label{eq:cubic-hard-sphere-ball}
 \boxed{\text{minimizing a cubic form over either the unit sphere or the
 unit ball is NP-hard.}}
\end{equation}
The convexity of the ball  simplifies the constraint handling, but
does not make the general problem computationally tractable.
The relation between spherical critical points and KKT points of the convex
ball formulation is also exact, apart from the unavoidable stationary point at
the origin.

\begin{proposition}[KKT points of the cubic ball problem]
\label{prop:cubic-ball-kkt}
Let $f$ be a nonzero cubic form and consider
\begin{equation}\label{eq:cubic-ball-problem}
 \min\{f(x):\|x\|^2\le1\}.
\end{equation}
Using the Lagrangian $L(x,\lambda)=f(x)+\lambda\,(\|x\|^2-1)$, a boundary
point $x\in\Sph^{n-1}$ is a KKT point of \eqref{eq:cubic-ball-problem} if and
only if it is a FONC point with $f(x)\le0$.  Its associated multiplier is
\begin{equation}\label{eq:cubic-ball-multiplier}
 0\,\leq\,\lambda^*\,=-\frac32f(x).
\end{equation}
At a negative boundary KKT point, the ball second-order necessary condition is exactly
the spherical SONC.  Every interior KKT point satisfies $\nabla f(x)=0$ and
$f(x)=0$; in particular, the origin is always a KKT point, but it is not a
local minimum of a nonzero cubic.
\end{proposition}

\begin{proof}
At a boundary KKT point, stationarity gives
$\nabla f(x)+2\lambda x=0$.  Taking the inner product with $x$ and using
Euler's identity gives $3f(x)+2\lambda=0$.  Thus $\lambda\ge0$ is equivalent
to $f(x)\le0$, and stationarity becomes $\nabla f(x)=3f(x)x$, which is
spherical FONC.  On tangent directions, the Lagrangian Hessian is
\begin{equation}\label{eq:ball-tangent-hessian}
 \left.(\nabla^2f(x)+2\lambda I)\right|_{x^\perp}
 =\left.(\nabla^2f(x)-3f(x)I)\right|_{x^\perp},
\end{equation}
the intrinsic spherical Hessian.  At an interior KKT point,
$\nabla f(x)=0$, and Euler's identity gives $f(x)=0$.  Finally, as nonzero
oddness gives both signs arbitrarily close to the origin, so $0$ is not a
local minimum.
\end{proof}

Consequently, for a benign cubic form every \emph{negative local minimizer} of the
ball problem is global.  It is straightforward to see that it also extends to all odd-degree forms.
Importantly,  when interpreting guarantees for local
algorithms, one must not replace ``local minimizer'' by arbitrary
``KKT point'' because KKT conditions also include saddle points and also the zero-valued origin.

\subsection{Concrete NP-hard quartic reductions}

Minimization of quartic forms on the sphere already contains several classical hard
problems.  The following reductions make the statement precise and complement
the previous hardness results for tensors cited in
\cite{LasserreOL2022,LasserreSIAM2022}.

First consider a graph $G$ with adjacency matrix $A_G$ and clique number
$\omega(G)$.  The Motzkin--Straus theorem \cite{MotzkinStraus1965} states
\begin{equation}\label{eq:motzkin-straus}
 \max_{z\ge0,\,\mathbf1^{\mathsf T}z=1}z^{\mathsf T}A_Gz
 =1-\frac1{\omega(G)}.
\end{equation}
(see for instance \cite{deKlerk2008Survey}).  Substituting $z_i=x_i^2$ maps the simplex onto the sphere and yields
\begin{equation}\label{eq:clique-quartic}
 \max_{\|x\|=1}\sum_{i,j}(A_G)_{ij}x_i^2x_j^2
 =1-\frac1{\omega(G)}.
\end{equation}
Consequently, minimizing the negative of the quartic in
\eqref{eq:clique-quartic} computes the maximum-clique number and is NP-hard.
Next, for a rational symmetric matrix $M$, define
\begin{equation}\label{eq:copositive-quartic}
 q_M(x)=\sum_{i,j}M_{ij}x_i^2x_j^2.
\end{equation}
The matrix $M$ is copositive if and only if
\begin{equation}\label{eq:copositive-sphere}
 \min_{\|x\|=1}q_M(x)\ge0.
\end{equation}
Indeed, $z_i=x_i^2$ converts the sphere into the simplex, and homogeneity
converts nonnegativity on the simplex into nonnegativity on the positive
orthant.  Murty and Kabadi \cite{MurtyKabadi1987} 
proved that deciding non-copositivity is NP-complete;
thus even the highly structured quartics
\eqref{eq:copositive-quartic} models an NP-hard spherical optimization problem.

Third, a biquadratic form $b(u,v)$ is a quartic form in the couple $(u,v)$.  Deciding its global nonnegativity is strongly NP-hard
\cite{AhmadiEtAl2013}; equivalently,
\begin{equation}\label{eq:biquadratic-sphere}
 b(u,v)\ge0\ \text{for all }(u,v)
 \quad\Longleftrightarrow\quad
 \min_{\|(u,v)\|=1}b(u,v)\ge0.
\end{equation}
Eq. \eqref{eq:biquadratic-sphere}
is the starting point for the strong NP-hardness of
deciding convexity of quartic polynomials in \cite{AhmadiEtAl2013}, a paper
explicitly cited in \cite{LasserreSIAM2022}.  Finally, an order-four symmetric
tensor $\mathcal A$ defines the quartic $x\mapsto\mathcal A[x,x,x,x]$;
optimizing this form and its negative computes its spectral norm, another
NP-hard tensor problem \cite{HillarLim2013}.

\section{A spectral-gap characterization}
In this section we sharpen the spectral characterizations of \cite{LasserreOL2022} 
reviewed in
\S 2.
\subsection{A global spectral characterization}
\label{subsec:gap-depth}

For a cubic form $f$ write
\[ H(x):=\nabla^2f(x),\qquad
 h(x):=\lambda_1(H(x))+\lambda_2(H(x))\,,\quad\forall x\in\R^n\,,
\]
and recall that $h$ is a concave function on \(\R^n\).

\begin{proposition}[A single two-eigenvalue SONC inequality]
\label{prop:ky-fan-sonc}
Let \(x\in\Sph^{n-1}\) be a FONC point of a cubic form $f$ with \(f(x)<0\). 
Then
\begin{equation}\label{eq:ky-fan-sonc}
 x\text{ satisfies SONC}
 \quad\Longleftrightarrow\quad
 h(x)\ge9f(x).
\end{equation}
Under these equivalent conditions,
\begin{equation}\label{eq:ky-fan-slack}
 h(x)-9f(x)
 =
 \lambda_{\min}\bigl(\nabla^2_S f(x)\bigr).
\end{equation}
In particular, the strict inequality in \eqref{eq:ky-fan-sonc} is equivalent
to SOSC.
\end{proposition}

\begin{proof}
The \emph{only if part} follows from Corollary \ref{cor:cubic-gap}. For the
\emph{if part}, let  \(r:=6f(x)<0\).  By FONC and Euler's identity, \(r\) is an eigenvalue of
\(H(x)\).  If an eigenvalue \(\mu\) were strictly smaller than \(r\), the sum
of the two smallest eigenvalues would satisfy
\[
 h(x)\le\mu+r<2r<\frac32r=9f(x),
\]
contradicting \eqref{eq:ky-fan-sonc}.  Hence \(r=\lambda_1(H(x))\), and
\(h(x)\ge9f(x)\) reduces to
\(\lambda_2(H(x))\ge3f(x)\), which is precisely Corollary \ref{cor:cubic-gap}.
The same calculation yields
\eqref{eq:ky-fan-slack}.
\end{proof}

The concavity of \(h\) does not turn cubic minimization into a convex
problem: the critical-point equations are nonconvex and \(h-9f\) is not
concave.  Moreover, maximizing \(h\) does not locate a global minimizer.  For
example, for \(f(x)=x_1^3\) in dimension \(n\ge3\), the global minimizer
\(-e_1\) has \(h(-e_1)=-6\), whereas every equatorial critical point has
\(h=0\).

\begin{corollary}[Global spectral characterization]\label{thm:global-lambda-one}
Let  \(f\in\cH_{n,3}\) be an arbitrary real cubic form. Then:
\begin{equation}\label{eq:global-lambda-one}
 \alpha(f):=\min_{\|x\|=1}\lambda_1(\nabla^2f(x))
 =6\min_{\|x\|=1}f(x)=6f_\star.
\end{equation}
The minimum defining \(\alpha(f)\) may equivalently be taken over the closed
unit ball $\mathcal{B}_n$.
\end{corollary}
\begin{proof}
Let \(T_f\in\operatorname{Sym}^3(\R^n)\) be the unique symmetric trilinear
form associated with \(f\), so that \(f(x)=T_f[x,x,x]\).  Then with $u,v\in\mathbb{S}^{n-1}$,
\[Df(x)[u]\,=\,3\,T_f(x,x,u)\,;\quad u^T\nabla^2f(x)\,v\,=\,D^2f(x)[u,v]\,=\,6\,T_f(x,u,v)\,,\]
and therefore $u^T\nabla^2f(x)u=6\,T_f(x,u,u)$.
As a result, for \(x\in\mathbb{S}^{n-1}\),
\[
 \lambda_1(\nabla^2f(x))\,=\,6\,\min_{\|v\|=1}T_f[x,v,v]\,,\]
which yields
 \[\min_{\Vert x\Vert=1}\lambda_1(\nabla^2f(x))
 \,=\,6\,\min_{\Vert x\Vert=\|v\|=1}T_f[x,v,v]\,=\,-6\,\max_{\Vert x\Vert=\|v\|=1}\vert \,T_f[x,v,v]\,\vert\,.\]
Next, by Banach's theorem for real symmetric tensors,
\[
 \max_{\|y\|=\|v\|=1}|T_f[y,v,v]|
 =
 \max_{\|x\|=1}|T_f[x,x,x]|\,=\, \max_{\|x\|=1}|f(x)|\,=\,-\min_{\|x\|=1}f(x)\,,
\]
where the odd degree of $f$ is invoked for the last equality.  This proves
\eqref{eq:global-lambda-one}.  Finally, for a cubic form,
\(\lambda_1(\nabla^2f(r\,x))=r\,\lambda_1(\nabla^2f(x))\) for all $r\geq0$,
which yields the conclusion on the ball formulation.
\end{proof}
So the global minimum of a nonzero cubic form on the unit sphere $\mathbb{S}^{n-1}$ 
(equivalently on the unit ball $\mathcal{B}_n$) is negative and is also 
one sixth of the global minimum of the smallest eigenvalue of the Hessian $\nabla^2f$ on
$\mathbb{S}^{n-1}$. Equivalently, both minima may be taken over the closed unit ball.

%
\begin{remark}
Corollary 2.3 characterizes SONC at negative critical points using the two smallest eigenvalues of the ambient Hessian. By contrast, 
Corollary \ref{thm:global-lambda-one} characterizes the global minimum value using only the smallest eigenvalue, minimized over the sphere.
However this latter characterization is concerned with the global minimum and not necessarily the global minimizers.
Indeed
\[\arg\min_{x\in\mathbb{S}^{n-1}} f(x)\,\subseteq\arg\min_{x\in\mathbb{S}^{n-1}} \lambda_1(\nabla^2f(x))\,,\]
but as the following example shows, the two sets may differ.
Let
$f(x):=x_1^3-3x_1x_2^2$. Then $f^*=-1$, and $\lambda_1(\nabla^2f(y))=6f^*$ for all $y\in\mathbb{S}^{1}$.
On the other hand, writing $x=(\cos(\theta),\sin(\theta))$,  the only points $\theta\in[0,2\pi]$ with
$f(x)=\cos(3\theta)=-1$ are $\theta\in\{\pi/3,\pi,5\pi/3\}$. So $\arg\min_{x\in\mathbb{S}^{n-1}} f(x)\not=
\arg\min_{x\in\mathbb{S}^{n-1}} \lambda_1(\nabla^2f(x))$.
\end{remark}
\subsection*{An odd-degree Hessian analogue of Rayleigh-Ritz}
From the arguments of the proof of Corollary \ref{thm:global-lambda-one}, it follows that the spectral result for cubic forms also holds for 
odd-degree forms (but not necessarily for even-degree  forms (except quadratic forms)). 
Indeed if \(T_f\) is the symmetric \(d\)-linear tensor associated with \(f\), then
\[v^{\mathsf T}\nabla^2f(x)v=d(d-1)T_f[x^{d-2},v,v],\]
and the result follows from oddness, the diagonal inclusion, and Banach's theorem.
Therefore  for an odd degree-$d$ form $f$,
\begin{equation}\label{eq:global-lambda-d}
 \alpha(f):=\min_{\|x\|=1}\lambda_1(\nabla^2f(x))
 =d(d-1)\min_{\|x\|=1}f(x)=d(d-1)f_\star\,.
\end{equation}
Optimization of homogeneous polynomials on the Euclidean unit sphere has a classical tensor-eigenvalue formulation: the constrained stationary points correspond to $Z$-eigenvectors of the associated symmetric tensor, and the global extrema are its extreme $Z$-eigenvalues; see \cite{Lim2005,Qi2005,KoldaMayo2011}. Equivalently, these stationary points arise from the homogeneous Rayleigh quotient $x\mapsto f(x)/\|x\|^d$, defined for $x\ne0$. 
Namely, for a degree-$d$ homogeneous polynomial $f$, differentiation yields
\[
\nabla\!\left(\frac{f(x)}{\|x\|^d}\right)
=\|x\|^{-d}\nabla f(x)
-df(x)\|x\|^{-d-2}x.
\]
On $\|x\|=1$, stationarity is therefore equivalent to \(\nabla f(x)=df(x)x\).
Writing $f(x)=T_f[x,\ldots,x]$ for its symmetric coefficient tensor, this becomes
$T_fx^{d-1}=f(x)x$, i.e., the $Z$-eigenvalue equation with unit normalization Qi \cite{Qi2005}; see also
Lim \cite{Lim2005} 
and Kolda--Mayo \cite{KoldaMayo2011}.
The classical (Z)-eigenvalue formulation identifies the spherical minimum with the smallest (Z)-eigenvalue 
of the associated symmetric tensor (and is valid for odd and even degrees). For odd degrees with $d\geq 3$, \eqref{eq:global-lambda-d} 
provides an alternative Hessian formulation to recover the same minimum (after dividing by $d(d-1)$) by minimizing the smallest 
ordinary matrix eigenvalue of the ambient Hessian over the sphere, without imposing the tensor-eigenvector equations. 
The identity follows from Banach's theorem for symmetric multilinear forms and crucially,  oddness; 
its minimizing points need \emph{not} be spherical critical points of the original form.



\begin{corollary}[Gap--depth characterization]\label{cor:gap-depth}
Let $\mathcal A_3\subset \mathcal H_{n,3}$ be the exceptional algebraic set in 
Huang's genericity result \cite[Theorem 3.4]{Huang2023}. 
Let $f\in\mathcal H_{n,3}$ and \(x\in\mathbb{S}^{n-1}\)  a FONC point satisfying SONC with $f(x)<0$.  

(i) Then:
\begin{eqnarray*}
 x\text{ is a global minimizer}
 &\Longleftrightarrow&
 \lambda_1(\nabla^2f(x))=\alpha(f)\,.\\
 x\text{ is a spurious minimizer}
 &\Longrightarrow&
 \lambda_1(\nabla^2f(x))>\alpha(f)\,,
\end{eqnarray*}
and the reverse implication $\Longleftarrow$ holds whenever $f\not\in\mathcal A_3$.

(ii) Consequently, if  $f\not\in\mathcal A_3$, 
negative spurious local minima are exactly the points satisfying
\begin{equation}\label{eq:negative-spurious-spectral}
 \nabla f(x)=3f(x)x,\quad f(x)<0,\quad
 h(x)\ge9f(x),\quad
 \lambda_1(H(x))>\min_{\|y\|=1}\lambda_1(H(y)).
\end{equation}
(iii) $f\not\in\mathcal A_3$
has negative spurious local minima if and only if
there exist $x,y\in\mathbb{S}^{n-1}$, such that $f(x)<0$, $f(y)<0$, and
\begin{eqnarray}\label{eq:negative-spurious-spectral-2}
 \nabla f(x)=3f(x)\,x,\quad h(x)\ge9f(x);&&
 \nabla f(y)=3f(y)\,y,\quad h(y)\ge9f(y)\,;\\
  \label{eq:negative-spurious-spectral-3}
  \lambda_1(H(x))&>&\lambda_1(H(y))\,.
  \end{eqnarray}
\end{corollary}

\begin{proof}
(i) follows from  Corollary \ref{thm:global-lambda-one}, and the fact that
if $f\not\in\mathcal A_3$, all SONC points are (strict) local minima.
(ii) If $f\not\in\mathcal A_3$ then all SONC points are strict local minima, and apply (i). To obtain (iii)
observe that by \eqref{eq:negative-spurious-spectral-3}, necessarily
$\lambda_1(H(x))>\min_{\|z\|=1}\lambda_1(H(z))$, and then invoke (ii).
\end{proof}

\begin{proposition}[Quartic spectral test]
\label{prop:quartic-sign}
Let $f$ be a quartic form and let $x\in\Sph^{n-1}$ satisfy FONC.  Write
$f=f(x)$ and let $\lambda_1\le\cdots\le\lambda_n$ be the eigenvalues of
$\nabla^2f(x)$.  The radial eigenvalue is $12f$.  
For a local minimum, SONC is equivalent to
\begin{equation}\label{eq:quartic-min-cases}
 \begin{cases}
  \lambda_1=12f,\quad \lambda_2\ge4f, & f<0,\\
  \nabla^2f(x)\succeq0\ \text{and }\lambda_1=0, & f=0,\\
  \lambda_1\ge4f, & f>0\,,
 \end{cases}
\end{equation}
while for a local maximum, SONC is equivalent to
\begin{equation}\label{eq:quartic-max-cases}
 \begin{cases}
  \lambda_n\le4f, & f<0,\\
  \nabla^2f(x)\preceq0\ \text{and }\lambda_n=0, & f=0,\\
  \lambda_n=12f,\quad \lambda_{n-1}\le4f, & f>0.
 \end{cases}
\end{equation}
In particular, a negative quartic local minimum satisfies
\begin{equation}\label{eq:quartic-lower-gap}
 \frac{\lambda_2-\lambda_1}{|\lambda_1|}\ge\frac23,
\end{equation}
whereas a positive quartic local maximum satisfies 
\begin{equation}\label{eq:quartic-upper-gap}
 \frac{\lambda_n-\lambda_{n-1}}{|\lambda_n|}\ge\frac23.
\end{equation}
At a negative FONC point with \(\lambda_1=12f\), strict inequality in \eqref{eq:quartic-lower-gap} is equivalent to SOSC for a minimum. At a positive FONC point with \(\lambda_n=12f\), strict inequality in \eqref{eq:quartic-upper-gap} is equivalent to SOSC for a maximum.
\end{proposition}

\begin{proof}
At FONC, \eqref{eq:euler} gives the radial eigenvalue $12f$, while every
tangent eigenvalue must be at least $4f$ for a minimum and at most $4f$ for a
maximum.  If $f<0$, then $12f<4f$; if $f>0$, then $12f>4f$.  Ordering the
radial and tangent eigenvalues yields \eqref{eq:quartic-min-cases} and
\eqref{eq:quartic-max-cases}.  The two gap thresholds are
$(4f-12f)/|12f|=2/3$ for $f<0$ and
$(12f-4f)/|12f|=2/3$ for $f>0$.
\end{proof}

The genericity result in \cite[Theorem 3.4]{Huang2023} states that, 
outside a proper algebraic exceptional set $\mathcal A_d\subset \mathcal H_{n,d}$ of
degree-$d$ forms, every feasible point satisfying FONC and SONC satisfies
SOSC and is a strict local minimizer.  Thus,
for a generic form, Corollary
\ref{cor:cubic-gap} and Proposition \ref{prop:quartic-sign} provide a complete local-minimum test, up to
the excluded equality degeneracies.  This is a statement about
\emph{local minimality} and absence of spurious minima additionally requires all
SONC points to have the same global value.

\begin{example}[A large gap need not mean globality]\label{ex:diagonal}
For
\begin{equation}\label{eq:diagonal}
 f(x)=\sum_{i=1}^n a_i x_i^3,
 \qquad 0<a_1<\cdots<a_n,
\end{equation}
every $-e_i$ is a strict local minimum with value $-a_i$, but only $-e_n$ is
global.  

At $-e_i$, $(\lambda_1(\nabla^2(-e_i)),\lambda_2(\nabla^2(-e_i)))=(-6a_i,0)$, and hence
$G(-e_i)=1$.  Therefore the $n-1$ spurious minima pass the gap test 
\eqref{eq:cubic-gap} with a large margin.
\end{example}

\begin{remark}[The non-benign class also has interior]
\label{rem:bad-open-set}
The local minima in Example \ref{ex:diagonal} are nondegenerate and have
distinct values.  Hence under arbitrary sufficiently small perturbations 
the coefficients in $\cH_{n,3}$, all of them persist, and their value ordering
persists.  The non-benign class therefore contains a nonempty open subset
of \(\cH_{n,3}\).  In particular, both the robust height family of
\cref{thm:robust-height} and this robust non-benign family are
full-dimensional.  Nonempty interior on either side does not determine which
class has larger normalized volume as \(n\) grows.
\end{remark}

\section{Benign and robustly benign forms}

We first state a result about how the benign class behaves with the dimension
$n$.
\subsection{Deterministic growth of the benign class}

With odd $d\geq 3$ fixed, let the set 
\begin{equation}\label{eq:benign}
 \mathcal G_{n,d}\,:=\{f\in\cH_{n,d}:\text{every spherical local minimum of $f$ is global}\}
 \end{equation}
be  the subset of degree-$d$ benign forms. 
For the phrasing ``$\mathcal G_{n,d}$ grows with $n$'' to be  meaningful, one need to
compare coefficient spaces, and a natural comparison is the zero-extension
\[\iota_n(\mathcal G_{n,d})\subsetneq \mathcal G_{n+1,d}\quad\mbox{with}\quad\iota_nf(x,z)\,:=f(x)\,,\quad\forall (x,z)\in\R^{n+1}\,.\]

\begin{theorem}[Dimension lifting]\label{thm:lifting}
(i) For every $n\ge2$,
\begin{equation}\label{eq:lifting}
 \iota_n(\mathcal G_{n,d})\subsetneq\mathcal G_{n+1,d}.
\end{equation}
(ii) $\mathcal G_{n,d}$ contains a nonempty open subset of $\cH_{n,d}$ and hence
has full parameter dimension $\binom{n+d-1}{d}$.
\end{theorem}
For a detailed proof the reader is referred to \S \ref{proof-thm-lifting}.

Theorem \ref{thm:lifting} shows that benignity is preserved under zero-extension;
it does not establish that odd-degree benign forms become more frequent as dimension grows. The coefficient-space dimensions change, and no probability measure on those spaces has been specified. 

\subsection{A characterization of benign forms}

Recall that with $\mathcal A_d\subset\mathcal H_{n,d}$
being the exceptional algebraic set in Huang's genericity result \cite[Theorem 3.4]{Huang2023},
a form $f\in\mathcal H_{n,d}\setminus\mathcal A_d$ is said to be \emph{generic}.
With $f\in\mathcal H_{n,d}$, at any  FONC point $x\in\mathbb{S}^{n-1}$, introduce the real symmetric matrix
\begin{equation}\label{eq:Mf}
 M_f(x):=
 \nabla^2f(x)-d\,f(x)I-d\,(d-2)\,f(x)xx^{\mathsf T}.
\end{equation}
Next, introduce the set:
\begin{equation}\label{eq:spectral-sonc-set}
 \mathcal S(f):= \{x\in\Sph^{n-1}:\:\nabla^2f(x)\,x\,=\,d(d-1)\,f(x)x,\ M_f(x)\succeq0\}\,.
\end{equation}
 \begin{theorem}
 \label{th:suff-cond}
With $d\geq2$, let $f\in\mathcal H_{n,d}$. 
Then at any FONC point $x\in\mathbb{S}^{n-1}$:
\begin{equation}
\label{eq:M_f-carac}
\mathrm{SONC}\quad\Longleftrightarrow \quad\nabla^2_Sf(x)\succeq0
\quad\Longleftrightarrow \quad M_f(x)\succeq0\,,\end{equation}
 independently of the sign of \(f(x)\).  In addition:

(i) The singleton condition $\mathcal S(f)=\{x_\star\}$ for odd-degree forms
is a deterministic sufficient certificate that \(f\) has a unique local minimizer,
hence a global minimizer . 

(ii) The condition $\mathcal S(f)=\{x_\star,-x_\star\}$ for even-degree forms
is a deterministic sufficient certificate that \(f\) has a unique antipodal pair $(x_\star,-x_\star)$
as local minimizer (hence as global minimizer). 

Finally for generic forms, in \eqref{eq:M_f-carac} one may replace SONC by ``$x$ is  local minimizer''.
\end{theorem}
\begin{proof}
As $x\in\mathbb{S}^{n-1}$ is a FONC point, the radial direction $u=x$ lies in \(\ker M_f(x)\). In addition,
$M_f(x)$ and $\nabla^2_Sf(x)$ define the same bilinear form
on $x^\perp$. Indeed, let $u,v\in x^\perp$. Then as $u^Tx=x^Tv=0$,
\[u^TM_f(x)v\,=\,u^T(\nabla^2f(x)-d\,f(x)\,I)v\,.\]
Next, as $\Pi_xu=u$ and $\Pi_xv=v$
\begin{eqnarray*}
u^T\nabla^2_Sf(x)v&=&u^T\Pi_x(\nabla^2f(x)-d\,f(x)\,I)\Pi_xv\\
&=&u^T(\nabla^2f(x)-d\,f(x)\,I)v\,=\,u^TM_f(x)v\,.
\end{eqnarray*}
In other words, and more precisely, on $\R^n=\mathrm{Span}\{x\}\oplus x^\perp$ one has the block decomposition
\[M_f(x)\,=\,\left[\begin{array}{cc}0&0\\0&\nabla^2_Sf(x)\end{array}\right]\,.\]
But at a FONC point, SONC being equivalent to $\nabla^2_Sf(x)\succeq0$, \eqref{eq:M_f-carac} follows.

The conditions (i) is  then an immediate consequences of \eqref{eq:M_f-carac} 
as a unique  SONC point is necessarily a global minimizer (as the latter always exists and is a SONC point)
and with a similar argument for (ii). The final statement follows from the fact that all SONC points are (strict) local minima 
whenever $f\not\in\mathcal A_d$.
\end{proof}
Observe that in view of \eqref{eq:M_f-carac},
\[\mathcal S(f)\,=\,\{\,x\in\mathbb{S}^{n-1}:\:\mbox{$x$ satisfies FONC and SONC}\,\}\]
because with $d\geq 2$, by Euler's identity, $\nabla^2f(x)x=d(d-1)x\Rightarrow \nabla f(x)=df(x)x$, and so
$x$ satisfies FONC, while by \eqref{eq:M_f-carac} it also satisfies SONC.

Notice that unlike the negative-value gap criterion  in Corollary \ref{cor:gap-depth}, 
formulation \eqref{eq:M_f-carac}  detects SONC points at every objective value, 
including positive or zero-valued spurious local minima.

Another nice feature of \eqref{eq:M_f-carac} is its  purely algebraic nature; indeed
via determinants of all principal minors, the semidefinite constraint  $M_f(x)\succeq0$ can be expressed through
polynomial inequality constraints. As a result, and since $\mathcal S(f)\neq\emptyset$, consider the set 
\[\Delta_f\,:=\,\{(x,y): x,y\in\,\mathcal S(f)\,;\quad f(x)<f(y)\,\}\,.\]
The condition $\Delta_f=\emptyset$, which is sufficient to assure that $f$ is benign,
has a Stengle certificate. See e.g. \cite[Theorem 2.11]{Lasserrebook}. Also observe that
for generic forms, 
\[\mbox{$f$ is benign}\quad \Longleftrightarrow \quad\Delta_f\,=\,\emptyset\,.\]

\subsection{Robustly benign forms}
A benign form is said to be \emph{robustly benign} if it remains benign after (possibly sufficiently small)
perturbation of its coefficients. A clear obstruction is when an odd-degree  benign form
has multiple nondegenerate isolated global minimizers since a sufficiently small generic perturbation preserves these minima while separating their values, and therefore only one of them remains global while the others become spurious.
Hence odd-degree benign forms with a unique isolated global minimizer are good candidate for being robustly benign. 
For even-degree forms, the corresponding obstruction is the existence of more than one distinct 
antipodal pair of nondegenerate global minimizers. 
The following conditions
\[\mathcal S(f)\,=\{\,x_\star\,\},
\qquad \ker M_f(x_\star)\,=\operatorname{span}\{\,x_\star\,\}\,,\]
or equivalently
\begin{equation}
\label{eq:odd-cond}
\mathcal S(f)\,=\{\,x_\star\,\},
\qquad \nabla^2_Sf(x_\star)\,\succ\,0\,,\end{equation}
provide  a useful sufficient certificate for an odd-degree form to be robustly benign.
For even-degree forms, the analogue conditions read
\begin{equation}
\label{eq:even-cond}
\mathcal S(f)\,=\{\,x_\star,-x_\star\,\},
\qquad \nabla^2_Sf(\pm x_\star)\,\succ\,0\,.\end{equation}

\begin{theorem}
\label{thm:cubic-robustness}
With $d$ odd, let $f\in\cH_{n,d}$ and $\mathcal S(f)$ be as in \eqref{eq:spectral-sonc-set}.  Suppose that there exists
$x_\star\in\Sph^{n-1}$ such that
\begin{equation}\label{eq:condition-54}
 \boxed{
 \mathcal S(f)=\{x_\star\},
 \qquad
 \nabla^2_S f(x_\star)\succ0}\,,
\end{equation}
or, equivalently,
\begin{equation}\label{eq:condition-54-M}
 \mathcal S(f)=\{x_\star\},
 \qquad
 \ker M_f(x_\star)=\mathrm{Span}\{x_\star\}.
\end{equation}
Then $f$ is robustly benign.  More precisely, for every coefficient
norm $\norm{\cdot}_{\mathrm{coef}}$ on $\cH_{n,d}$, there exists
$\varepsilon>0$ such that every $\widetilde f\in\cH_{n,d}$ satisfying
\begin{equation}\label{eq:coefficient-neighborhood}
 \norm{\widetilde f-f}_{\mathrm{coef}}<\varepsilon
\end{equation}
has exactly one local minimizer $x(\widetilde f)\in\mathbb{S}^{n-1}$.  This point is
nondegenerate and is the unique global minimizer of $\widetilde f$.
Moreover,
\[
 x(\widetilde f)\longrightarrow x_\star
 \qquad\text{as}\qquad
 \norm{\widetilde f-f}_{\mathrm{coef}}\longrightarrow0.
\]
\end{theorem}

\begin{proof}
The proof proceeds into  four steps.

\medskip
\noindent Step 1: $x_\star$ is the unique local and global minimizer
of $f$.
Every local minimizer satisfies FONC and SONC and therefore belongs to
$\mathcal S(f)$.  Since
$\mathcal S(f)=\{x_\star\}$, the point $x_\star$ is the only possible local
minimizer.  A global minimizer exists by compactness of the sphere, and every
global minimizer is a local minimizer.  Hence $x_\star$ is the unique global
minimizer.  The strict inequality
$\nabla^2_S f(x_\star)\succ0$ makes it a nondegenerate strict global minimum.

\noindent
Step 2: 
Since \(\nabla^2_Sf(x_\star)\succ0\), choose \(\mu>0\) and a sufficiently small strongly geodesically 
convex ball \(U\) centered at \(x_\star\) such that
\[v^T\nabla^2_Sf(x)\,v\,\ge\,3\mu\,\|v\|^2\,,\]
for all \(x\in\overline U\) and \(v\in x^\perp\). Then every sufficiently small \(C^2\) perturbation \(g\) of $f$  satisfies
\[v^T\nabla^2_S(f+g)(x)\,v\,\ge2\,\mu\,\|v\|^2\,\mbox{on \(\overline U\).}\]
Hence \(f+g\) is strongly geodesically convex on \(U\) and has at most one critical point there.
Choose local coordinates \(\psi:V\subset\mathbb R^{n-1}\to U\) with \(\psi(0)=x_\star\), and define
\[\mathcal F(g,z)=\nabla_z\,((f+g)\circ\psi)(z)\,.\]
Then $\mathcal F(0,0)=0$ and $D_z\mathcal F(0,0)=\nabla_z^2(f\circ\psi)(0)$,
which represents the positive-definite intrinsic Hessian
\(\nabla^2_Sf(x_\star)\). By the implicit-function theorem, for every sufficiently small perturbation \(g\), 
there exists a critical point \(x(g)=\psi(z(g))\) near \(x_\star\) of $f+g$. It is the unique critical point in \(U\) and, 
because its intrinsic Hessian is positive definite, it is a nondegenerate strict local minimum.

\medskip
\noindent Step 3: uniform exclusion of local minima away from
$x_\star$.
Let
\[
 K:=\Sph^{n-1}\setminus U.
\]
This set is compact and contains no point of $\mathcal S(f)$.  Define the
continuous function
\begin{equation}\label{eq:exclusion-function}
 \Phi_f(x):=
 \max\left\{
  \norm{\nabla_S f(x)},
  -\lamin\bigl(\nabla^2_S f(x)\bigr)
 \right\},
 \qquad x\in K.
\end{equation}
At every $x\in K$, one has $\Phi_f(x)>0$.  Indeed, if
$\nabla_S f(x)\ne0$, then the first term is positive.  If
$\nabla_S f(x)=0$, then $x\notin\mathcal S(f)$ forces
$\nabla^2_S f(x)\not\succeq0$, so
$\lamin(\nabla^2_S f(x))<0$ and the second term is positive.  By compactness and continuity,
one obtains 
\begin{equation}\label{eq:eta-margin}
 \eta:=\min_{x\in K}\Phi_f(x)>0.
\end{equation}
Next, suppose that
\begin{equation}\label{eq:C2-smallness-exterior}
 \sup_{x\in\Sph^{n-1}}
 \norm{\nabla_S\widetilde f(x)-\nabla_S f(x)}\,<\,\frac\eta3
\quad\mbox{and}\quad
 \sup_{x\in\Sph^{n-1}}
 \norm{\nabla^2_S\widetilde f(x)-\nabla^2_S f(x)}_{\mathrm{op}}
 \,<\,\frac\eta3.
\end{equation}
Fix $x\in K$.  By \eqref{eq:eta-margin}, at least one of the following
alternatives holds.
If $\norm{\nabla_S f(x)}\ge\eta$, then
\[
 \norm{\nabla_S\widetilde f(x)}
 \ge\norm{\nabla_S f(x)}-\frac{\eta}{3}
 \ge\frac{2\eta}{3}>0,
\]
so $x$ is not a critical point of $\widetilde f$. On the other hand, if $-\lamin(\nabla^2_S f(x))\ge\eta$, 
then by Weyl's eigenvalue inequality 
\[
\lambda_{\min}\bigl(\nabla_S^2\widetilde f(x)\bigr)
\le
\lambda_{\min}\bigl(\nabla_S^2f(x)\bigr)
+
\left\|
\nabla_S^2\widetilde f(x)-\nabla_S^2f(x)
\right\|_{\mathrm{op}}\,.
\]
Combining with \eqref{eq:C2-smallness-exterior} yields
\[
 \lamin(\nabla^2_S\widetilde f(x))
 \le\lamin(\nabla^2_S f(x))+\frac{\eta}{3}
 \le-\frac{2\eta}{3}<0.
\]
Thus $x$ is not a SONC point for $\widetilde f$ and cannot be a local minimizer.
Consequently, $\widetilde f$ has no local minimizer on $K$.
Combining Steps 2 and 3, $x(\widetilde f)$ is the unique spherical local
minimizer of $\widetilde f$.  A global minimizer exists by compactness and
is necessarily a local minimizer.  It must therefore be
$x(\widetilde f)$, which proves that $\widetilde f$ is benign and has a
unique global minimizer.

\medskip
\noindent
Step 4: from $C^2$ perturbations to coefficient
perturbations. This is similar to step 4 in the proof of Theorem \ref{thm:lifting} in
\S \ref{proof-thm-lifting} and is omitted.
\end{proof}

\begin{remark}
(a) Condition \eqref{eq:condition-54} provides with a directly usable deterministic
sufficient test.  The theorem does not assert that every robustly benign
cubic must satisfy \eqref{eq:condition-54}.  Establishing a converse would
require a separate analysis of whether every additional degenerate SONC
point can be converted into a spurious local minimum by an arbitrarily small
cubic coefficient perturbation.

(b) A similar result holds for quartic (and even-degree) forms. 
But every even-degree form, including every
even-degree coefficient perturbation, satisfies
$F(-x) = F(x)$. Hence the two continued critical points are antipodal and have the same value.
This antipodal symmetry creates an additional technicality in the proof. 

\end{remark}

\section{Explicit benign cubic families}

\subsection{A full-dimensional robust height neighborhood}
\label{subsec:robust-height}

The qualitative interior assertion in Theorem \ref{thm:lifting} admits an explicit
and dimension-independent intrinsic criterion.  
With \(a\in\R^n\setminus\{0\}\) is associated the \emph{height} cubic form 
$f_0(x):=(a^Tx)\Vert x\Vert^2$ and the height family of cubics
$\{(a^Tx)\Vert x\Vert^2\}_{a\in\R^n\setminus\{0\}}$.
Next, let \(q\in\cH_{n,3}\) be arbitrary, and consider the cubic form
\begin{equation}\label{eq:height-perturbation}
 x\mapsto f(x)\,:=(a^{\mathsf T}x)\|x\|^2+q(x)\,.
\end{equation}
On the unit sphere the reference function (or height function) is  \(x\mapsto f_0(x):=a^{\mathsf T}x\). Of course if  $q=0$ then  minimizing 
$f$ on $\mathbb{S}^{n-1}$ is equivalent to minimizing $f_0$ on $\mathbb{S}^{n-1}$ (hence 
on $\mathcal B_n$),  a convex optimization problem. However, when $q\neq0$
the problem is not convex anymore. Define
\begin{equation}\label{eq:height-deltas}
 \delta_1:=
 \sup_{x\in\Sph^{n-1}}\|\nabla_S q(x)\|,
 \qquad
 \delta_2:=
 \sup_{x\in\Sph^{n-1}}\|\nabla^2_S q(x)\|_{\mathrm{op}}.
\end{equation}

\begin{theorem}[Explicit robust height family]\label{thm:robust-height}
If
\begin{equation}\label{eq:height-condition}
 \delta_1^2+\delta_2^2<A^2,
\end{equation}
then $f$ in \eqref{eq:height-perturbation} has exactly one local minimizer
\(x_\star\in\mathbb{S}^{n-1}\), hence its unique global minimizer.  In addition, letting 
\( \kappa:=\sqrt{A^2-\delta_1^2}\)\quad\mbox{with $A=\Vert a\Vert$}\,,
one has
\begin{equation}\label{eq:height-cap-curvature}
 a^{\mathsf T}x_\star\le-\kappa,
 \qquad
 \nabla^2_S f(x_\star)\,
 \succeq\,(\kappa-\delta_2)I_{x_\star^\perp}\succ0.
\end{equation}
In addition, the conclusion is stable under \emph{arbitrary} sufficiently small cubic
coefficient perturbations, and not just small perturbations that preserve 
the cubic family $\{(a^Tx)\Vert x\Vert^2\}_{0\neq a\in\R^n}$.
\end{theorem}

\begin{proof}
On the sphere let \(s:=a^{\mathsf T}x\).  The height function satisfies
\[
 \nabla_S(a^{\mathsf T}x)\,=\,\Pi_xa,\qquad
 \nabla^2_S(a^{\mathsf T}x)\,=\,-sI_{x^\perp}\,.
\]
If \(x\) is a local minimizer, FONC yields
\(\Pi_xa\,=\,-\nabla_S q(x)\), hence using $a=(a^Tx)x+\Pi_xa$,
\[
 A^2=s^2+\|\Pi_xa\|^2\le s^2+\delta_1^2
\]
and therefore \(|s|\ge\kappa\).  Next, SONC yields
\( -sI+\nabla^2_S q(x)\succeq0\), and for a unit tangent vector \(v\),
\[
 0\le-s+v^{\mathsf T}\nabla^2_S q(x)v\le-s+\delta_2,
\]
so that \(s\le\delta_2\).  Condition \eqref{eq:height-condition} is equivalent to
\(\kappa>\delta_2\), and consequently the alternative \(s\ge\kappa\) is
impossible.  Every local minimizer therefore belongs to
\[
 \mathcal C_-:=\{x\in\Sph^{n-1}:a^{\mathsf T}x\le-\kappa\}\,.
\]
On the unit sphere, the geodesic distance between two points 
\(c,x\in\mathbb S^{n-1}\) is $d_S(c,x)=\arccos(c^{\mathsf T}x)$,
and \(\arccos\) is decreasing on \([-1,1]\).
So $\mathcal C_-$ is a convex spherical cap of radius
\(\arccos(\kappa/A)<\pi/2\), hence strongly geodesically convex. Next,
\[
 \nabla^2_S f(x)
 =-(a^{\mathsf T}x)I+\nabla^2_S q(x)
 \succeq(\kappa-\delta_2)I\succ\,0\,,\quad\forall x\in\mathcal C_-\,,\]
i.e., $f$ is strictly (and even strongly) geodesic convex, which combined
with strong geodesic convexity of $\mathcal C_-$, implies that 
$\mathcal C_-$ contains at most one critical
point.  A global minimizer exists and, being a local minimizer, lies in the
cap.  As a result it is  the unique local (hence global) minimizer.
\end{proof}
Observe that 
the singleton sufficient condition in Theorem \ref{th:suff-cond}(i)
is satisfied and \eqref{eq:height-cap-curvature} is an instance
of  the sufficient condition for robustness stated in \eqref{eq:odd-cond}.
In addition, \eqref{eq:height-condition} is an explicit condition on the form $q$
for $f$ to be benign or for $f_0$ to be robustly benign.

Importantly, we emphasize that small enough perturbations 
(controlled by \eqref{eq:height-condition})
provide benign cubic forms $f$ in \eqref{eq:height-perturbation} that 
are outside the original height family 
$\{(a^Tx)\Vert x\Vert^2\}_{a\in \R^n\setminus\{0\}}$ of cubic forms. \\

Also the intrinsic condition \eqref{eq:height-condition} has a direct certificate in terms of tensor coefficients.  Write
\( q(x)=E[x,x,x]\) for a symmetric third-order tensor and set
\begin{equation}\label{eq:betaE}
 \beta(E):=
 \sup_{\|x\|=\|u\|=\|v\|=1}|E[x,u,v]|.
\end{equation}
Since
\[
 \nabla q(x)=3E[x,x,\cdot]\,,\quad
 \nabla^2q(x)=6E[x,\cdot,\cdot]\,,\mbox{ and }
 \nabla^2_S q(x)\, =
 \left.\Pi_x\bigl(\nabla^2q(x)-3q(x)I\bigr)
 \Pi_x\right|_{x^\perp},
\]
one obtains \( \delta_1\le3\beta(E)\) and \(\delta_2\le9\beta(E)\).
\begin{corollary}[An explicit tensor ball]\label{cor:height-tensor-ball}
For
\[
 f(x)=(a^{\mathsf T}x)\|x\|^2+E[x,x,x],
\]
either of the sufficient conditions
\begin{equation}\label{eq:height-tensor-bounds}
 \beta(E)<\frac{\|a\|}{\sqrt{90}},
 \qquad\text{or (the more conservative)}\qquad
 \|E\|_{\mathrm F}<\frac{\|a\|}{\sqrt{90}},
\end{equation}
guarantees that $f$ has a unique spherical local and global minimizer.
\end{corollary}

\begin{proof}
The first inequality implies
\(\delta_1^2+\delta_2^2\le90\beta(E)^2<A^2\).
The second follows from \(\beta(E)\le\|E\|_{\mathrm F}\).
\end{proof}

In Corollary \ref{cor:height-tensor-ball} 
the ball around the tensor \(f_0\) associated with $(a^Tx)\Vert x\Vert^2$ is full-dimensional.
The perturbation \(E\) may contain every
cubic monomial.  At the unperturbed minimizer \(x_\star=-a/\|a\|\), the
spectrum of the ambient Hessian begins with 
\(\lambda_1(\nabla^2f_0(x_\star)),\lambda_2(\nabla^2f_0(x_\star))=(-6A,-2A)\), so
the (local) smallest-eigenvalue gap in \eqref{eq:G} reads
\begin{equation}\label{eq:height-gap}
 G(x_\star)=\frac23>\frac12\,,
\end{equation}
which  is consistent with Corollary \ref{cor:cubic-gap}, but does not by itself prove
that $f_0$ is robustly benign: the
global exclusion of other minima comes from \eqref{eq:height-tensor-bounds}.

\begin{remark}[All odd degrees]\label{rem:odd-degree-height}
Theorem \ref{thm:robust-height} is not specific to cubics.  For odd \(d=2k+1\), consider
\[
 f_d(x)=(a^{\mathsf T}x)\|x\|^{d-1}+q_d(x),
 \qquad q_d\in\cH_{n,d}.
\]
Its restriction to the sphere is \(a^{\mathsf T}x+q_d(x)\).  With
\(\delta_1,\delta_2\) defined as in \eqref{eq:height-deltas} for \(q_d\), the
same proof as that of Theorem \ref{thm:robust-height}
shows that
\[
 \delta_1^2+\delta_2^2<\|a\|^2
\]
implies a unique spherical local and global minimizer.  Thus the 
height family $\{(a^\mathsf Tx)\Vert x\Vert^{d-1}\}_{0\neq a\in\R^n}$ 
is our first example  of a degree-independent odd-homogeneous
family that is robustly benign.
\end{remark}

\subsection{Eikonal cubics}

The following  identity 
\begin{equation}\label{eq:eikonal}
 \|\nabla f(x)\|^2=9\,c^2\,\|x\|^4  \qquad(x\in\R^n)
\end{equation}
is the eikonal equation  for cubic forms studied by Tkachev \cite{Tkachev2010} 
in order to classify its solutions; it provides with  a particularly transparent
class for our purpose.

\begin{theorem}[Eikonal cubics have no spurious minima]\label{thm:eikonal}
Suppose a cubic form satisfies \eqref{eq:eikonal} 
for some $c>0$.  Then every spherical critical value is either $-c$ or $c$.
Consequently every local minimum is global with value $-c$.
\end{theorem}

\begin{proof}
On the unit sphere, Euler's identity yields
\begin{equation}\label{eq:grad-eikonal}
 \|\nabla_S f\|^2
 =\|\nabla f-3f(x)x\|^2=9\,(c^2-f^2).
\end{equation}
Thus $|f|\le c$ and in addition,  the spherical critical points are exactly the level sets
$f=\pm c$.  The first level is globally minimal and the second globally
maximal.
\end{proof}
Also consider the cubic
\begin{equation}\label{eq:chebyshev}
(t,y)\mapsto  f(t,y)\,:=c\bigl(t^3-3t\|y\|^2\bigr)
\end{equation}
which satisfies \eqref{eq:eikonal}. As on the sphere $t^2+\Vert y\Vert^2=1$, it equals
$c(4t^3-3t)=cT_3(t)$ where $T_3$ is the degree-$3$  Chebyshev polynomial of the first kind.
This example illustrates  that ``no spurious minima'' does not
mean ``single minimizer'': Indeed besides the south pole, the 
points $(t,y)$ with $t=1/2$ and $\|y\|=\sqrt3/2$ 
are also global minimizers.  Tkachev's
classification \cite{Tkachev2010} shows that, after scaling and rotation,
solutions of \eqref{eq:eikonal} are of the type \eqref{eq:chebyshev} or one of the four
exceptional Cartan cubics in dimensions $5,8,14,26$.

\subsection{An axial family with exact benign criterion}

Inspired by the eikonal cubics \eqref{eq:chebyshev}, consider the family of cubic forms 
\begin{equation}\label{eq:axial}
 f(t,y)=t\bigl(\alpha t^2+y^{\mathsf T}By\bigr),
 \qquad \alpha>0,\quad B=B^{\mathsf T}\succ0,
\end{equation}
on the sphere $\{(t,y): t^2+\|y\|^2=1\}$. Let us call it the \emph{axial family}
as it singles out the single variable $t$ (axis $e_1$).

The axial familes and reducible eikonal example \eqref{eq:chebyshev} share a same spirit: both
consist of cubics given by a linear form times a quadratic form.
However, as we next see, they are genuinely different families.  The eikonal
example \eqref{eq:chebyshev} corresponds to
\[
 \alpha=c,\qquad B=-3cI\prec0,
\]
whereas \eqref{eq:axial} assumes $B\succ0$.  Even 
with $B\succeq0$ \eqref{eq:axial} would not include \eqref{eq:chebyshev}.  Conversely,
the spectral inequality \eqref{eq:axial-threshold} below is a benign-landscape
criterion; it neither forces $B$ to be scalar nor implies the rigid eikonal
identity \eqref{eq:eikonal}.  Thus, under the hypotheses used here, neither
family should be viewed as a specialization of the other.

\begin{theorem}[Sharp threshold for benign  axial]
\label{thm:axial}
The cubic (axial) form \eqref{eq:axial} has no spurious spherical local minimum if and only if
\begin{equation}\label{eq:axial-threshold}
 \lambda_{\min}(B)\le\frac{3\alpha}{2}.
\end{equation}
More precisely:
\begin{enumerate}
 \item the north pole $(1,0)$ is a strict, necessarily spurious, local minimum
 if and only if $\lambda_{\min}(B)>3\alpha/2$; at equality it is not a local
 minimum;
 \item the south pole $(-1,0)$ is a local minimum if and only if
 $\lambda_{\max}(B)\le3\alpha/2$, and is then global;
 \item if $\lambda_{\max}(B)>3\alpha/2$, the global minima
 are
 \begin{equation}\label{eq:axial-minimizers}
  t=-\sqrt{\frac{\beta}{3(\beta-\alpha)}},
  \quad y\in E_\beta,
  \quad \|y\|^2=1-t^2,
  \qquad \beta=\lambda_{\max}(B).
 \end{equation}
\end{enumerate}
\end{theorem}
 For a detailed proof the interested reader is referred to \S \ref{proof-thm-axial}.

\begin{corollary}[Full-coefficient robustness]
\label{cor:axial-robust}
Suppose that the axial cubic \eqref{eq:axial} satisfies
\begin{equation}\label{eq:strict-axial-robust}
 \lambda_{\max}(B)<\frac{3\alpha}{2}.
\end{equation}
Then, for every fixed norm \(\|\cdot\|_{\rm coef}\) on \(\cH_{n,3}\), there
exists \(\varepsilon>0\) such that every cubic form \(\widetilde f\) satisfying
\[
 \|\widetilde f-f\|_{\rm coef}<\varepsilon
\]
has exactly one spherical local minimum, and that point is its unique global
minimum.  The perturbation is arbitrary in \(\cH_{n,3}\); in particular,
\(\widetilde f\) need not have the axial form \eqref{eq:axial}.
\end{corollary}

\begin{proof}
Under \eqref{eq:strict-axial-robust}, equation
\eqref{eq:lagrange}
has no nonpolar solution.  Thus the only critical
points of \(f|_{\Sph^{n-1}}\) are the two poles $p_+=(1,0)$ and $p_-(-1,0)$.  Their intrinsic Hessians are
\[\nabla^2_S f(1,0)\,=\,
 2B-3\alpha I\prec\,0\,;
  \qquad
 \nabla^2_S f(-1,0)\,=\,3\alpha I-2B\succ\,0\,.
\]
They are therefore a nondegenerate maximum and a nondegenerate minimum,
respectively.  With similar arguments as those used in the proof of Theorem \ref{thm:robust-height}
but now using the condition \eqref{eq:strict-axial-robust}, the perturbed cubic thus has exactly one local minimum
(hence global  minimum).
\end{proof}

\begin{remark}[The height cubic as a strict axial member]
Recall the height cubic form  \(f(x)=(a^{\mathsf T}x)\|x\|^2\).
After a rotation such that \(a=Ae_1\), where \(A=\|a\|>0\),
\[
 f(t,y)=A\,t(t^2+\|y\|^2),
\]
a member of the axial family \eqref{eq:axial} with  \(\alpha=A\) and \(B=AI\).  Consequently
\(\lambda_{\max}(B)=A<3A/2\), and by Corollary \ref{cor:axial-robust},
one retrieves that the height cubic  is robustly benign.
More generally, Corollary \ref{cor:axial-robust} provides 
robust benign neighborhoods around every strict axial member, not merely
around this isotropic example.  
A generic nearby cubic loses the factorization
\((a^{\mathsf T}x)(x^{\mathsf T}Qx)\) but remains benign.
\end{remark}

Both Theorem \ref{thm:axial} and Corollary \ref{cor:axial-robust} 
illustrate the useful role of the Hessian linear pencil: a global
landscape statement becomes a comparison of the extremal eigenvalues of one
symmetric matrix $B$ with the scalar threshold $3\alpha/2$.

\subsection{A general cubic pencil family}

While after an orthogonal change of variables, the axial family reduces to a
constant diagonal matrix and a scalar parameter, one now introduces
a cubic family for which several coefficient matrices need \emph{not} commute. Namely,
let $k\ge2$, $m\ge1$, write $t\in\R^k$, $y\in\R^m$, and let
$B_1,\ldots,B_k\in\R^{m\times m}$ be symmetric.  Define
\begin{equation}\label{eq:noncommuting-pencil}
 B(t)=\sum_{i=1}^k t_iB_i,
 \qquad
 \rho_B=\sup_{\|v\|=1}\|B(v)\|_{\rm op},
\end{equation}
and consider the cubic form
\begin{equation}\label{eq:noncommuting-cubic}
 f(t,y)=t_1(\|t\|^2+\|y\|^2)+y^{\mathsf T}B(t)y.
\end{equation}

\begin{theorem}[A noncommuting-pencil benign family]
\label{thm:noncommuting-pencil}
If $\rho_B\le1/2$, the restriction of \eqref{eq:noncommuting-cubic} to
$\Sph^{k+m-1}$ has exactly two critical points,
\begin{equation}\label{eq:noncommuting-poles}
 x_-=(-e_1,0),\qquad x_+=(e_1,0).
\end{equation}
They are its unique global minimum and maximum, with values $-1$ and $1$, and
therefore $f$ has no spurious local minimum (i.e. is benign).

If $\rho_B<1/2$, both extrema are nondegenerate. In addition, 
in the full space of homogeneous cubics in $k+m$ variables, 
$f$ belongs to the interior of the
class having exactly one local minimum and one local maximum.
\end{theorem}

\begin{proof}
On the sphere $\mathbb{S}^{k+m-1}$, $\Vert t\Vert^2+\Vert y\Vert^2=1$.
Set $u=t_1$ and $q(t,y)=y^{\mathsf T}B(t)y$  so that $f=u+q$ on $\mathbb{S}^{m+k-1}$.
At every critical point $\nabla_Sf(t,y)=0$ and in particular
$\langle\nabla_S f,\nabla_S u\rangle=0$. We next prove that 
whenever $|u|<1$ and $\rho_B\le1/2$, $\langle\nabla_S f,\nabla_S u\rangle\neq0$,
and therefore \emph{no} nonpolar point can be a critical point of $f$.

Observe that $\nabla_S u=e_1-ux$ and since $q$ is homogeneous of degree three,
\begin{align}
 \langle\nabla_S f,\nabla_S u\rangle
 &=1-u^2+\partial_{t_1}q-3uq \notag\\
 &=1-u^2+y^{\mathsf T}B(e_1-3ut)y.       \label{eq:height-derivative}
\end{align}
By linearity and using \eqref{eq:noncommuting-pencil},  
$\|B(v)\|_{\rm op}\le\rho_B\|v\|$.  Moreover,
\begin{equation}\label{eq:pencil-vector-bound}
 \|e_1-3ut\|^2
 =1-6u^2+9u^2\|t\|^2
 \le1+3u^2,
 \qquad
 \|y\|^2\le1-u^2.
\end{equation}
Indeed
\begin{eqnarray*}
\|e_1-3ut\|^2&=&(e_1-3ut)^{\mathsf T}(e_1-3ut)\\
&=&\|e_1\|^2-6u\,e_1^{\mathsf T}t+9u^2\|t\|^2.
\end{eqnarray*}
Since $\|e_1\|=1$ and $e_1^{\mathsf T}t=t_1=u$,
one obtains $\|e_1-3ut\|^2=1-6u^2+9u^2\|t\|^2$.
Finally, as \((t,y)\in\mathbb S^{k+m-1}\),  $\|t\|^2\le1$ and
therefore,
\[1-6u^2+9u^2\|t\|^2\,\le  1-6u^2+9u^2\,=\,1+3u^2.\]
Consequently,
\begin{equation}\label{eq:height-monotonicity}
 \langle\nabla_S f,\nabla_S u\rangle
 \ge(1-u^2)\bigl(1-\rho_B\sqrt{1+3u^2}\bigr)>0
\end{equation}
whenever $|u|<1$ and $\rho_B\le1/2$.  
Thus no nonpolar point is critical.
At the two poles all first derivatives of $q$ vanish because $y=0$, so the
poles are critical.  Compactness now implies that they are the unique global
extrema.

At $x_-$ and $x_+$, respectively, the intrinsic  Hessians in tangent
coordinates $(t_2,\ldots,t_k,y)$ are
\begin{equation}\label{eq:noncommuting-spherical-hessians}
 \nabla^2_S f(x_-)=I_{k-1}\oplus(I_m-2B_1),
 \qquad
 \nabla^2_S f(x_+)=-I_{k-1}\oplus(-I_m+2B_1).
\end{equation}
They are positive and negative definite when $\rho_B<1/2$.  Choose disjoint
neighborhoods of the poles on which these signs persist.  The spherical
gradient has a positive minimum on their compact complement.  Every
sufficiently small $C^2$ perturbation therefore has no critical point on that
complement, while the implicit function theorem continues exactly one
critical point near each pole.  Coefficient convergence of fixed-degree forms
implies $C^2$ convergence on the sphere, proving openness in the full cubic
coefficient space.  In particular the allowed perturbations need not retain
the partially quadratic structure of \eqref{eq:noncommuting-cubic}.
\end{proof}

In Theorem \ref{thm:noncommuting-pencil} there is no simultaneous diagonalization result in disguise.  For
$k=m=2$, take
\begin{equation}\label{eq:pauli-pencil}
 B_1=\varepsilon
 \begin{pmatrix}1&0\\0&-1\end{pmatrix},
 \qquad
 B_2=\varepsilon
 \begin{pmatrix}0&1\\1&0\end{pmatrix},
 \qquad 0<\varepsilon\le\frac12.
\end{equation}
Then $[B_1,B_2]\ne0$, $\rho_B=\varepsilon$, and
\begin{align}
 f={}&t_1(t_1^2+t_2^2+y_1^2+y_2^2) \notag\\
 &+\varepsilon\bigl[t_1(y_1^2-y_2^2)+2t_2y_1y_2\bigr].
 \label{eq:pauli-cubic}
\end{align}
The matrices of the pencil $A(t)=t_1I+B(t)$ have eigenvalues
\begin{equation}\label{eq:pauli-eigenvalues}
 \lambda_\pm(A(t))=t_1\pm\varepsilon\|t\|.
\end{equation}
Their eigenvectors vary with $t$, and no constant orthogonal transformation
diagonalizes both coefficient matrices.  However, Theorem \ref{thm:noncommuting-pencil} states
that \eqref{eq:pauli-cubic} has a unique
local and global minimum.
Finally, the family fits the spectral characterization
\eqref{eq:cubic-gap} in Corollary \ref{cor:cubic-gap}.  Put
\begin{equation}\label{eq:noncommuting-beta}
 \beta=\max\{0,\lambda_{\max}(B_1)\}.
\end{equation}
At $x_-$ the ambient Hessian is
\begin{equation}\label{eq:noncommuting-ambient-hessian}
 \nabla^2f(x_-)
 =[-6]\oplus[-2I_{k-1}]\oplus[-2I_m-2B_1].
\end{equation}
Since $k\ge2$ and $0\le\beta\le\rho_B\le1/2$,
\begin{equation}\label{eq:noncommuting-gap}
 \lambda_1=-6,\qquad
 \lambda_2=-2-2\beta,\qquad
 G(x_-)=\frac{2-\beta}{3}
 \ge\frac{2-\rho_B}{3}\ge\frac12.
\end{equation}
The inequality is strict when $\rho_B<1/2$, consistently with SOSC.  At the
endpoint equality can occur although the global minimum remains unique; this
is one of the nongeneric degeneracies excluded by Huang's genericity result \cite{Huang2023}.

\begin{remark}[The open class is larger than the pencil family]
The last assertion of Theorem \ref{thm:noncommuting-pencil} is essential.  The
matrices $B_i$ parametrize a comparatively small algebraic subset of the full
space of cubic coefficients.  If $\rho_B<1/2$, however, there exists
$\delta(f)>0$ such that every cubic $\widetilde f$ satisfying
\begin{equation}\label{eq:ambient-cubic-ball}
 \|\widetilde f-f\|_{C^2(\Sph^{k+m-1})}<\delta(f)
\end{equation}
has exactly one spherical local minimum and one local maximum, both global.
The perturbation in \eqref{eq:ambient-cubic-ball} is completely unrestricted:
it may contain any cubic monomials, need not remain quadratic in $y$, and need
not be representable by the original pencil.  Since for fixed degree and dimension, 
coefficient norms and $C^2$-norms are equivalent, the same statement
holds for a sufficiently small perturbation in the full
coefficient space. 

However, one  recognizes that this family remains a relatively small perturbation of a height function. Its significance is chiefly that the matrix pencil is genuinely noncommuting and that the benignity conclusion is stable in the full coefficient space. 
\end{remark}

\begin{corollary}[Convex-ball formulation and local descent]
\label{cor:noncommuting-ball-algorithm}
Under the strict hypothesis $\rho_B<1/2$, the cubic $f$ in
\eqref{eq:noncommuting-cubic} has exactly two KKT points for minimization on
the Euclidean unit ball:
\begin{equation}\label{eq:noncommuting-ball-kkt}
 0\qquad\text{and}\qquad x_-=(-e_1,0).
\end{equation}
The point $x_-$ is the unique global minimizer, whereas $0$ is not a local
minimum.  Therefore any feasible descent algorithm whose objective values are
nonincreasing, which starts from a point $x_0$ with $f(x_0)<0$ and converges to
a KKT point, necessarily converges to $x_-$.  The same conclusion holds for
every sufficiently small ambient cubic coefficient perturbation as in 
\eqref{eq:ambient-cubic-ball}, with its associated  global minimizer replacing
$x_-$.
\end{corollary}

\begin{proof}
By Theorem \ref{thm:noncommuting-pencil}, the only critical points are
$x_-$ and $x_+$, with values $-1$ and $1$.  By
Proposition \ref{prop:cubic-ball-kkt}, $x_-$ is the only boundary KKT point.  If a nonzero
interior point $z$ satisfied $\nabla f(z)=0$, homogeneity would imply that
$z/\|z\|$ is a spherical critical point of value zero, a contradiction.
Thus the origin is the sole interior KKT point.  A nonincreasing sequence
starting with a negative value cannot converge to the zero-valued origin, so
any KKT limit must be $x_-$.  The perturbation statement follows from the
same argument.
\end{proof}

\begin{remark}[No hidden ambient convexity on the negative region]
The conclusion of Corollary \ref{cor:noncommuting-ball-algorithm} is \emph{not} a consequence
of convexity of the objective, even after restricting attention to the region
$\{x:f(x)<0\}$ in which the descent argument operates.  For every cubic form,
$\nabla^2f(-x)=-\nabla^2f(x)$, so no nonzero cubic form can be convex on a
centrally symmetric neighborhood of the origin.  For the present family this
failure is especially explicit at its global minimizer: under $\rho_B<1/2$,
the matrix in \eqref{eq:noncommuting-ambient-hessian} is negative definite.
Thus the key behind Corollary is indeed 
a benign KKT landscape on a convex feasible set.
\end{remark}

The ball is a convex feasible set and its Euclidean projection is explicit,
\begin{equation}\label{eq:ball-projection}
 \Pi_{\cB_n}(z)=\frac{z}{\max\{1,\|z\|\}},
\end{equation}
so projected-gradient, trust-region, and related feasible methods are
easily implementable.
Moreover, a
negative initialization is easy to obtain: by oddness, replacing any
$z$ with $f(z)\ne0$ by the appropriate sign provides with  a point with negative value.
Thus, for the robustly benign family, the convex
ball formulation converts convergence to a KKT point under descent from a
negative starting point into a global-optimization guarantee.  It does not turn an
arbitrary cubic problem into a convex optimization problem.

\section{A quartic family with an exact benignity criterion}

Let $C=C^{\mathsf T}\in\R^{(n-1)\times(n-1)}$, $\rho\in\R$, and  consider
the family of quartic forms
\begin{equation}\label{eq:quartic-family}
 (t,y)\mapsto f_{\rho,C}(t,y)\,:
 =\rho(t^2+\|y\|^2)^2+t^2y^{\mathsf T}Cy\,,
 \end{equation}
 indexed by $(\rho,C)$. For any real scalar $\tau$, let  $E_\tau(C):=\mathrm{Ker}(C-\tau\,I)$.

\begin{theorem}[Semidefinite-coupling characterization]
\label{thm:quartic-family}
The quartic $f_{\rho,C}$ has no spurious local minima on $\Sph^{n-1}$ if and
only if
\begin{equation}\label{eq:quartic-benign-iff}
 C\succeq0\qquad\text{or}\qquad C\preceq0.
\end{equation}
Writing $\kappa=\lambda_{\min}(C)$, its global minimum is
\begin{equation}\label{eq:quartic-global-value}
 f_\star\,=\,\min_{\Sph^{n-1}}f_{\rho,C}
 =\rho+\frac14\min\{0,\kappa\}\,.
\end{equation}
If $C\succeq0$, the global minimizers are the spherical points for which
$t=0$ or $y\in\ker C$.  If $\kappa<0$, they are exactly
\begin{equation}\label{eq:quartic-global-points}
 t^2=\frac12,\qquad \|y\|^2=\frac12,\qquad
 y\in E_\kappa.
\end{equation}
When $C$ is indefinite, every point $(0,v)$ with
$v^{\mathsf T}Cv>0$ provides a nonstrict spurious local minimizer with value $\rho$.
\end{theorem}

\begin{proof}
Set
\begin{equation}\label{eq:quartic-reduction}
 s=t^2,\qquad y=\sqrt{1-s}\,v,\qquad
 0\le s\le1,\quad\|v\|=1.
\end{equation}
On the sphere, define the coupling term \(I_C(t,y) := t^2y^TC\,y= s(1-s)v^TC\,v\),
so that \(f_{\rho,C}= \rho+ I_C(t,y)\). Next, an interior local minimum with nonzero coupling must have $s=1/2$.
Then positive coupling has negative curvature in the $s$
direction and cannot be locally minimal. For negative coupling, local minimality in 
$v$ requires a minimum of the Rayleigh quotient, hence membership in the eigenspace associated with
the lowest eigenvalue.

$\bullet$ If \(C\succeq0\), then \(I_C (t,y)\geq0\), and every point at 
which \(I_C(t,y) = 0\)--equivalently, \(t= 0\) or \(y\in\mathrm{Ker}\,C\)--is a global minimizer 
with value \(\rho\).

$\bullet$ Suppose that $C\preceq0$ and $C\ne0$, and set
\( \kappa:=\lambda_{\min}(C)<0\). Then
\(I_C(t,y)\le0\) and since \( 0\le s(1-s)\le\frac14\) and
 \(v^{\mathsf T}Cv\ge\kappa\), the smallest possible value is
\( \min I_C=\frac{\kappa}{4}\).
It is attained exactly when \( s=\frac12\), \(v\in E_\kappa\),
where recall that $E_\kappa$ is the eigenspace of $C$ associated with $\kappa$.
Equivalently, the global minimizers satisfy
\begin{equation}\label{eq:negative-global-points}
 t^2=\frac12,
 \qquad
 \|y\|^2=\frac12,
 \qquad
 y\in E_\kappa,
\end{equation}
and the global value $\rho+\frac{\kappa}{4}<\rho$.

In addition, every zero-coupling point has value \(\rho\), strictly larger 
than \(f_\star=\rho+\kappa/4\), and more importantly, no such point is a local minimum. 
Indeed, let \(u\) be a unit \(\kappa\)-eigenvector. At either pole, the perturbation \(y=\varepsilon u\) provides with
a coupling \((1-\varepsilon^2)\varepsilon^2\kappa<0\). At a point \((t_0,y_0)\) with \(t_0\neq0\) and \(y_0\in\ker C\), the spherical perturbation
\[
y_\varepsilon=y_0+\varepsilon u,\qquad
t_\varepsilon=\operatorname{sign}(t_0)\sqrt{t_0^2-\varepsilon^2}
\]
provides with a coupling \((t_0^2-\varepsilon^2)\varepsilon^2\kappa<0\).

Finally, consider equatorial points: Case 1: \(v^{\mathsf T}Cv<0\)
then define the spherical curve
\(t(\theta)=\sin\theta\), \(y(\theta)=\cos\theta\,v\).
It satisfies \(t(\theta)^2+\|y(\theta)\|^2=\sin^2\theta+\cos^2\theta=1$,
and \((t(\theta),y(\theta))\longrightarrow(0,v)\) as \(\theta\to0\).
Along this curve, the coupling term is
\[\mathcal I_C(t(\theta),y(\theta))\,=\,
t(\theta)^2y(\theta)^{\mathsf T}Cy(\theta)\,=\,
\sin(\theta)^2\cos(\theta)^2\,v^{\mathsf T}Cv\,.\]
As \(v^{\mathsf T}Cv<0\), it follows that \(I_C(t(\theta),y(\theta))<0\)
for sufficiently small nonzero \(\theta\). Therefore
\(f_{\rho,C}(t(\theta),y(\theta))<\rho=f_{\rho,C}(0,v)\), and so \((0,v)\) is not a local minimum.

Case 2: \(v^{\mathsf T}Cv=0\) and so  \(v\in\ker C\).
Since \(C\neq0\) and \(C\preceq0\), there exists a unit eigenvector \(u\) with
\(Cu=\kappa u\), \(\kappa=\lambda_{\min}(C)<0\).
Because \(v\in\ker C\) and \(C\) is symmetric,
\(u^{\mathsf T}v=0\). 
Rotate \(v\) by
\(v(\theta)=\cos\theta\,v+\sin\theta\,u\).
Then \(\|v(\theta)\|=1\) and \(v(\theta)\longrightarrow v\) as \(\theta\to0\).
Moreover,
\[v(\theta)^{\mathsf T}Cv(\theta)\,=\,
\cos(\theta)^2\,v^{\mathsf T}Cv
+2\sin\theta\cos\theta\,v^{\mathsf T}Cu
+\sin(\theta)^2\,u^{\mathsf T}Cu\,=\,
\kappa\sin^2\theta<0\]
for every nonzero sufficiently small \(\theta\).
Now define the spherical curve
\(t(\theta)=\sin\theta\), \(y(\theta)=\cos\theta\,v(\theta)\).
Again,
\(t(\theta)^2+\|y(\theta)\|^2=1\) and
\((t(\theta),y(\theta))\longrightarrow(0,v)\).
Along this curve,
\[\mathcal I_C(t(\theta),y(\theta))\,=\,\sin(\theta)^2\cos(\theta)^2\,
v(\theta)^{\mathsf T}Cv(\theta)\,=\,\kappa\sin(\theta)^4\cos(\theta)^2\,<0.\]
Hence
\(f_{\rho,C}(t(\theta),y(\theta))=\rho+\kappa\sin(\theta)^4\cos(\theta)^2\,<\rho\).
Thus an equatorial kernel point is not a local minimum either. 

$\bullet$ Indefinite $C$.
Suppose that \(\|v\|=1\) and
\(v^{\mathsf T}Cv>0\).
At the  equatorial point
\((0,v)\), \(I_C(0,v)=0\) and \(f_{\rho,C}(0,v)=\rho\).
Because the map
\(u\longmapsto u^{\mathsf T}Cu\) is continuous and is strictly positive at \(v\), 
there is a neighborhood \(U\) of \(v\) on the unit sphere such that
\(u^{\mathsf T}Cu>0\) for every \(u\in U\).
For a point near \((0,v)\), write
\(t^2=s\) and \(y=\sqrt{1-s}\,u\),
where \(s\ge0\) is small and \(u\in U\). Then
\(I_C(t,y)=s(1-s)u^{\mathsf T}Cu \ge0\).
Therefore
\(f_{\rho,C}(t,y)\ge\rho\,=\,f_{\rho,C}(0,v)\)
throughout a neighborhood of \((0,v)\). Hence \((0,v)\) is a nonstrict local minimum.
But because \(C\) is indefinite,
\(\kappa=\lambda_{\min}(C)<0\). Choosing
\(s=\frac12\) and \(u\in E_\kappa\) yields
\(f_{\rho,C}=\rho+\frac{\kappa}{4}<\rho\) and so
 the equatorial local minimum \((0,v)\) is spurious.
\end{proof}

\begin{example}[Positive benign and indefinite non-benign quartics]
On $\Sph^2$, the positive definite quartic
\begin{equation}\label{eq:quartic-positive-example}
 2(t^2+y_1^2+y_2^2)^2-t^2(4y_1^2+y_2^2)
\end{equation}
has minimum value $1$ and exactly four strict global minima
$t=\pm1/\sqrt2$, $y_1=\pm1/\sqrt2$, $y_2=0$.  Replacing the second term by
$t^2(-4y_1^2+y_2^2)$ makes $C$ indefinite: the global value remains $1$, but
$(0,0,\pm1)$ are nonstrict spurious minima of value $2$.
\end{example}

\begin{remark}[Robustness inside the family \eqref{eq:quartic-family}]
\label{rem:quartic-robust}
If \(C\succ0\), then every sufficiently small symmetric-matrix perturbation \(\widetilde C\) still satisfies
\(\widetilde C\succ0\).
Similarly, if \(C\prec0\), every sufficiently small matrix perturbation remains negative definite. 
Since Theorem \ref{thm:quartic-family} states that every semidefinite \(C\) gives a benign landscape, definite matrices are interior points of the benign set relative to the parameter space \((\rho,C)\).
Thus there is structured robustness (i.e. robustness within the family):
\[
C\succ0\ \text{or}\ C\prec0
\quad\Longrightarrow\quad
f_{\widetilde\rho,\widetilde C}
\text{ remains benign for nearby }(\widetilde\rho,\widetilde C).
\]
\end{remark}

\subsection{A robustly benign quartic}
While the benign quartics of the family \eqref{eq:quartic-family} may be robust
relative to the parameter space $(\rho,C)$ (see Remark \ref{rem:quartic-robust})
they are not necessarily robust to perturbation of their coeffcients in $\mathcal H_{n,4}$.
Nevertheless such robustly benign quartics do exist.

\begin{proposition}[A full coefficient neighborhood of benign quartics]
\label{prop:robust-quartic-anchor}
Let $a>0$ and
\begin{equation}\label{eq:robust-quartic-anchor}
 x\mapsto Q_a(x)\,:=-a x_1^2\|x\|^2.
\end{equation}
There exists $\delta_a>0$ such that every quartic form $\widetilde Q$ with
\begin{equation}\label{eq:robust-quartic-ball}
 \|\widetilde Q-Q_a\|_{C^2(\Sph^{n-1})}<\delta_a
\end{equation}
has no spurious spherical local minimum.  Its global-minimum set consists of
one antipodal pair near $\{\pm e_1\}$, and these minima are nondegenerate.
The perturbation is arbitrary in the full quartic coefficient space and need
not retain the axial form \eqref{eq:robust-quartic-anchor}.
\end{proposition}

\begin{proof}
On the sphere, $Q_a(x)=-a x_1^2$.  Its only minima are the nondegenerate
antipodal pair $\{\pm e_1\}$.  Its equator $x_1=0$ is a critical manifold,
but in the normal tangent direction its spherical Hessian has eigenvalue
$-2a$; hence every equatorial point has a uniform direction of negative
curvature.  Between fixed small neighborhoods of the poles and the equator,
the spherical gradient has a positive minimum.  These three strict margins
persist under a sufficiently small $C^2$ perturbation.  The implicit function
theorem continues one nondegenerate local minimum near each pole, no point
near the equator can satisfy SONC, and no other point can be critical.
Every quartic is even, so the two continued minima are antipodal and have the
same value.  They are therefore both global and no spurious local minimum
exists.
\end{proof}

As in the cubic case, equivalence of norms on the finite-dimensional quartic
coefficient space converts \eqref{eq:robust-quartic-ball} into a full
coefficient ball.  A generic member of this ball contains monomials that are absent
in \eqref{eq:robust-quartic-anchor}; it lies outside the axial family while
remaining benign.  Robustness therefore enlarges a tractable structured
model into an open ambient class.

\section{Conclusion}

Generically, the quite simple Hessian's spectral characterization of SONC points in \cite{LasserreOL2022}
becomes that of local minima (and minimizers) thanks to Huang's result \cite{Huang2023} on genericity.  
However, it cannot, alone, distinguish a global minimum from a spurious one.
One has first provided several results that supply the missing global comparison, and
used them to prove that some families of cubic and quartic forms are benign
and even robustly benign to small perturbations on the full
space of their coefficients,
 even when the resulting form (after perturbation) leaves the family.

Moreover, and interestingly,  under dimension lifting, 
benign cubic families grow deterministically and contain an open set
(in the space of the coefficients)
at every dimension. Therefore it would be quite interesting to determine what is the relative
proportion of benign cubic forms to all cubic forms, and similarly for quartic forms.
This  is a topic of further investigation beyond the scope of the present paper.

\subsection*{Acknowledgement}
The author acknowledges the use of OpenAI ChatGPT to assist
with brainstorming and drafting, and  is solely responsible for the final content and conclusions.

\section{Appendix}

\subsection{Proof of Theorem \ref{thm:lifting}}
\label{proof-thm-lifting}
\begin{proof}
We separate the proof into four steps:

\medskip
\noindent
$\bullet$ Step 1: the zero-extension preserves the benign property.
Let $f\in\cG_{n,d}$ and set $F:=\iota_nf$.  A nonpolar point of
$\Sph^n\subset\R^{n+1}$ can be written as
\[
 (x,z)=(ru,z),
 \qquad
 u\in\Sph^{n-1},\quad 0<r\le1,\quad r^2+z^2=1\,,
\]
and homogeneity yields
\begin{equation}\label{eq:lifted-value}
 F(ru,z)=f(ru)=r^df(u).
\end{equation}
If $f\equiv 0$ then so is $F$ and the conclusion is trivial. 
Next, let $f\neq0$;  then its global minimum is strictly negative ($f_\star<0$). We examine three cases.

1. No local minimum with \(0<r<1\)
Suppose \((ru,z)\) were a local minimum. If \(f(u)>0\) then  decreasing \(r\) decreases \(r^df(u)\).
If \(f(u)<0\) then increasing \(r\) decreases \(r^df(u)\).
If \(f(u)=0\), varying $r$ provides no information, but varying \(u\) shows that \(u\) must be a local minimizer of \(f\). 
But since \(f\) is benign, this would imply $f(u)=f_\star<0$, in contradiction with \(f(u)=0\).
Hence no point with \(0<r<1\) can be a local minimum of \(F\).

2. A local minimum on the equator is global.
Let \((u,0)\), corresponding to \(r=1\), be a local minimum of \(F\). Variations along the equator show that \(u\) is a local minimum of \(f\). Since \(f\) is benign, $f(u)=f_\star<0$. For every \(v\in\mathbb S^{n-1}\) and \(0\le r\le1\),
$r^df(v)\ge f_\star$. Indeed: if \(f(v)\ge0\), then \(r^df(v)\ge0>f_\star\); if \(f(v)<0\), then $r^df(v)\ge f(v)\ge f_\star$,
because \(0\le r^d\le1\). Thus $F(ru,z)=r^df(v)\ge f_\star=F(u,0)$ and so \((u,0)\) is a global minimum of \(F\).

3. The poles are not local minima. At either pole \(r=0\), $F=0$.
Let \(u_\star\) be a global minimizer of \(f\). Since \(f_\star<0\), points arbitrarily close to either pole with 
small \(r>0\) and direction \(u_\star\) satisfy $F(ru_\star,z)=r^df_\star<0$.
Hence neither pole is a local minimum. 

Therefore every local minimum of the lifted form lies on the equator and comes 
from a global  minimum of \(f\). Hence all local minima are global, i.e., $\iota_n(\mathcal G_{n,d})
\subseteq
\mathcal G_{n+1,d}$.

%
%
%
\medskip
\noindent
$\bullet$ Step 2: the inclusion is strict.
As $d$ is odd, the form \(G(x,z):=z\,\|(x,z)\|^{d-1}\)
is a polynomial because $d-1$ is even.  On $\Sph^n$ it is simply the height
$z$, and hence it has only one local minimum, the global minimum at the south
pole.  Thus $G\in\cG_{n+1,d}$, whereas $G$ is not a zero-extension. This proves strictness in \eqref{eq:lifting}.

\medskip
\noindent
$\bullet$ Step 3: the odd-degree interior point. This the most technical part.
Let $d=2p+1$ and set
\begin{equation}\label{eq:odd-anchor}
 f_0(x):=x_1\|x\|^{2p}.
\end{equation}
On the sphere, $f_0(x)=s(x):=x_1$, and therefore:
\begin{equation}\label{eq:height-identities}
 \nabla_Sf_0(x)\,=\,\nabla_S s(x)=e_1-sx,
 \qquad
 \nabla^2_Sf_0(x)\,=\,\nabla^2_S s(x)=-sI_{x^\perp}.
\end{equation}
Thus the only critical points are $e_1$ and $-e_1$,
\[
 \nabla^2_ Sf_0(e_1)=-I_{e_1^\perp},
 \qquad
 \nabla^2_S f_0(-e_1)=I_{e_1^\perp}.
\]
They are respectively a nondegenerate maximum and a nondegenerate minimum.

Choose disjoint coordinate neighborhoods $U_+$ and $U_-$ of the poles and
smaller neighborhoods $V_+\Subset U_+$ and $V_-\Subset U_-$.  Applied in
local charts on $\Sph^{n-1}$, the parameter-dependent implicit-function
theorem gives a $C^2$-neighborhood $\mathcal U_0$ of $f_0$ such that every
$f\in\mathcal U_0$ has exactly one critical point in each $V_\pm$, and these
critical points remain respectively a strict maximum and a strict minimum.
To justify the word ``exactly'' on the whole of $U_\pm$, observe that the
compact annuli $\overline{U_\pm\setminus V_\pm}$ contain no critical point of
$f_0$.  Hence $\|\nabla_S f_0\|$ has a positive minimum there, and after
shrinking $\mathcal U_0$ no critical point of $f$ can occur in those annuli.

Finally, the set \( K:=\Sph^{n-1}\setminus(U_+\cup U_-)\)
is compact and contains no critical point of $f_0$.  Therefore
\[
 \eta:=\min_{x\in K}\|\nabla_S f_0(x)\|>0.
\]
If $\|\nabla_S(f-f_0)\|_{C^0}<\eta/2$, no critical point of $f$ lies in $K$.
Consequently every sufficiently small $C^2$ perturbation of $f_0$ has exactly
one local minimum, which is necessarily global.  Hence
$f_0\in\cG_{n,d}$.

\medskip
\noindent
$\bullet$ Step 4: From the $C^2$ topology to coefficients.
Write
\[
 q(x)=\sum_{|\alpha|=d}q_\alpha x^\alpha,
 \qquad
 \|q\|_{\mathrm{coef}}
 :=\left(\sum_{|\alpha|=d}q_\alpha^2\right)^{1/2}\,,
\]
and, for a restriction to the sphere, define
\begin{equation}\label{eq:spherical-C2-norm}
 \|q\|_{C^2(\Sph^{n-1})}
 :=
 \max_{x\in\Sph^{n-1}}
 \max\bigl\{
  |q(x)|,\ \|\nabla_S q(x)\|,\
  \|\nabla^2_S q(x)\|_{\rm op}
 \bigr\}.
\end{equation}

For every multi-index $\beta$ with $|\beta|\le2$, let $(\alpha)_\beta:=\prod_{i=1}^n
\frac{\alpha_i!}{(\alpha_i-\beta_i)!}$ where $(\alpha)_\beta=0$ unless $\beta\le\alpha$ componentwise.  Then:
\[
 \partial^\beta q(x)
 =\sum_{|\alpha|=d}(\alpha)_\beta q_\alpha x^{\alpha-\beta}\,,
\]
and since
$|x_i|\le1$ on the sphere, by Cauchy--Schwarz, there exist constants
$C_{n,d,\beta}<\infty$ such that
\[
 \sup_{x\in\Sph^{n-1}}|\partial^\beta q(x)|
 \le C_{n,d,\beta}\|q\|_{\mathrm{coef}}\,,
\]
so  that $\|q\|_{C^2(\Sph^{n-1})}
 \le C_{n,d}\|q\|_{\mathrm{coef}}$ for some finite $C_{n,d}$.  
Therefore each $C^2$ neighborhood obtained in
Steps~3--4 contains a Euclidean coefficient ball in $\cH_{n,d}$.  Hence
$\cG_{n,d}$ has nonempty coefficient-space interior.  Any subset of a
finite-dimensional vector space containing an open ball has the full ambient
parameter dimension $\binom{n+d-1}{d}$.
\end{proof}

\subsection{Proof of Theorem \ref{thm:axial}}
\label{proof-thm-axial}
\begin{proof}
One proceeds in 6 steps.

\medskip
\noindent
$\bullet$ Step 1: diagonalization and the two pole Hessians.
An orthogonal change of the \(y\)-coordinates which diagonalizes \(B\), 
does not change the sphere and the local-minimum problem.  Let \(v\) be an arbitrary  unit
eigenvector of \(B\), with \(Bv=bv\).  Near the north pole $p_+$, use the spherical curve
\[ t=\sqrt{1-r^2},  \qquad  y\,=rv\,.\]
Because \(y^{\mathsf T}By=br^2\),
\begin{align}
 f\bigl(\sqrt{1-r^2},rv\bigr)
 &=
 \sqrt{1-r^2}\bigl(\alpha(1-r^2)+br^2\bigr) \notag\\
 &=
 \sqrt{1-r^2}\bigl(\alpha+(b-\alpha)r^2\bigr).
 \label{eq:north-exact}
\end{align}
The expansion
\[
 \sqrt{1-r^2}
 =
 1-\frac{r^2}{2}-\frac{r^4}{8}+O(r^6)
\]
yields
\begin{equation}\label{eq:north-expansion}
 f\bigl(\sqrt{1-r^2},rv\bigr)
 =
 \alpha
 +\left(b-\frac{3\alpha}{2}\right)r^2
 +\left(\frac{3\alpha}{8}-\frac b2\right)r^4
 +O(r^6).
\end{equation}
Hence in the tangent direction \(v\), the eigenvalue of $\nabla^2_Sf(p_+)$ (the intrinsic Hessian at \(p_+\))
is \( 2b-3\alpha\).
Since this holds in every eigendirection $v$ of \(B\), one obtains
\begin{equation}\label{eq:north-hessian}
 \nabla^2_S f(p_+)\,=2B-3\alpha I.
\end{equation}
At the south pole $p_-$, with similar arguments (now using \(t=-\sqrt{1-r^2}\) and
$ y=rv$), one obtains
\begin{equation}\label{eq:south-hessian}
 \nabla^2_S f(p_-)=3\alpha I-2B.
\end{equation}

It follows immediately that \(p_+\) is a strict nondegenerate local minimum
when \(\lmin(B)>3\alpha/2\), and that it is not a local minimum when
\(\lmin(B)<3\alpha/2\).  Similarly, \(p_-\) is a strict nondegenerate local
minimum when \(\lmax(B)<3\alpha/2\).

\medskip
\noindent
$\bullet$ Step 2: the equality cases at the poles.
Suppose that \(b=3\alpha/2\).  In the associated eigendirection,
\eqref{eq:north-expansion} becomes
\begin{equation}\label{eq:north-equality}
 f\bigl(\sqrt{1-r^2},rv\bigr)
 =
 \alpha-\frac{3\alpha}{8}r^4+O(r^6).
\end{equation}
Hence with every sufficiently small nonzero \(r\), the resulting value is strictly
smaller than \(f(p_+)=\alpha\) and so $p_+$ is not a local
minimum at equality (its intrinsic Hessian is only positive semidefinite).
Similarly, 
At the south pole, and with the same eigendirection,
\begin{equation}\label{eq:south-equality}
 f\bigl(-\sqrt{1-r^2},rv\bigr) = -\alpha+\frac{3\alpha}{8}r^4+O(r^6)\,,
\end{equation}
and a similar result holds.
The global argument in Step~5 below shows that \(p_-\) is a global minimum
whenever \(\lmax(B)\le3\alpha/2\), including equality.

\medskip
\noindent
$\bullet$ Step 3: equations for a nonpolar critical point.
The derivatives of \(f\) read:
\begin{equation}\label{eq:ambient-derivatives}
 \frac{\partial f}{\partial t}
 =
 3\alpha t^2+y^{\mathsf T}By,
 \qquad
 \nabla_y f=2tBy.
\end{equation}
The Lagrange multiplier equations for the sphere are
\begin{equation}\label{eq:lagrange}
 3\alpha t^2+y^{\mathsf T}By=\lambda t,
 \qquad
 2tBy=\lambda y,
 \qquad
 t^2+\|y\|^2=1
\end{equation}
for a nonpolar point, \(y\ne0\).  Moreover, \(t\ne0\) otherwise the second
equation in \eqref{eq:lagrange} gives \(\lambda=0\), whereas the first yields 
\(y^{\mathsf T}By=0\), in contradiction with \(B\succ0\).

Next, by \eqref{eq:lagrange}, \(By=\frac{\lambda}{2t}y\), i.e.,
\(y\) belongs to some eigenspace \(E_b\) of \(B\).  Writing
\(By=by\), we obtain
\[
 \lambda=2bt,
 \qquad
 y^{\mathsf T}By=b\|y\|^2=b(1-t^2).
\]
The first equation in \eqref{eq:lagrange} becomes \(3\alpha t^2+b(1-t^2)=2bt^2\),
or equivalently, \(b=3(b-\alpha)t^2\).
Therefore every nonpolar critical point satisfies
\begin{equation}\label{eq:critical-t}
 \boxed{\displaystyle
 t^2=\frac{b}{3(b-\alpha)}.}
\end{equation}
Recall that $B\succ0$ and so $b>0$. Hence the requirement \(0<t^2<1\) is equivalent to
\( b>\frac{3\alpha}{2}\).
Indeed, positivity in \eqref{eq:critical-t} requires \(b>\alpha\), and
\(t^2<1\) is equivalent to \(b<3(b-\alpha)\), i.e.,
\(2b>3\alpha\).  Finally, at \(b=3\alpha/2\), by
\eqref{eq:critical-t}, \(t^2=1\), and so hence \(y=0\), which is a pole,
not a nonpolar point.

\medskip
\noindent
$\bullet$ Step 4: second-order classification of nonpolar critical points.
Fix a unit vector \(v\in E_b\) and restrict \(f\) to 
\( y=\sqrt{1-t^2}\,v\).
The resulting one-variable function is
\begin{equation}\label{eq:meridian}
 t\mapsto \phi_b(t)\, =
 t\bigl(\alpha t^2+b(1-t^2)\bigr)
 =
 bt+(\alpha-b)t^3.
\end{equation}
Its first two derivatives are
\begin{equation}\label{eq:meridian-derivatives}
 \phi_b'(t)=b+3(\alpha-b)t^2,
 \qquad
 \phi_b''(t)=6(\alpha-b)t.
\end{equation}
At a nonpolar critical point \(b>3\alpha/2>\alpha\).  Hence
\(\alpha-b<0\), and by \eqref{eq:meridian-derivatives}
\[
 \phi_b''(t)<0\quad\text{when }t>0,
 \qquad
 \phi_b''(t)>0\quad\text{when }t<0.
\]
Thus with $t>0$, a nonpolar critical point has a direction of negative
curvature and cannot be a local minimum.  

So one need consider
the \(t<0\) branch. 
Let \(w\perp y\) be an eigenvector
of \(B\) with eigenvalue \(c\).  Then \((0,w)\) is tangent to the sphere at
\((t,y)\).  Since the Lagrange multiplier is \(\lambda=2bt\), 
\begin{equation}\label{eq:transverse-curvature}
 (0,w)^{\mathsf T}
 \bigl(\nabla^2f-\lambda I\bigr)(0,w)
 =
 2t(c-b)\|w\|^2.
\end{equation}
For \(t<0\), nonnegativity of \eqref{eq:transverse-curvature} for every
eigenvalue \(c\) requires \(c\le b\).  Therefore a \(t<0\)
nonpolar critical point can satisfy SONC only if
\begin{equation}\label{eq:top-eigenvalue}
 b=\beta:=\lmax(B).
\end{equation}
If \(\beta\) has multiplicity larger than one, the zero-curvature directions
within \(E_\beta\) are tangent to the resulting manifold of minimizers.

\medskip
\noindent
$\bullet$ Step 5: identification of the global minimizers.
For \(t\ge0\), the assumptions \(\alpha>0\) and \(B\succ0\) imply
\(f(t,y)\ge0\), whereas \(f(-1,0)=-\alpha<0\).
Hence every global minimizer must have \(t<0\).

For fixed \(t<0\), minimizing \(f(t,y)\) means maximizing
\(y^{\mathsf T}By\), because multiplication by \(t<0\) reverses the
inequality.  Under \(\|y\|^2=1-t^2\), and with $\beta=\lmax(B)$,
\( y^{\mathsf T}By\le\beta(1-t^2)\)
with equality precisely when \(y\in E_\beta\).  Hence the global minimum problem
reduces to
\begin{equation}\label{eq:scalar-global}
 \min_{-1\le t\le0}
 \psi_\beta(t),
 \qquad
 \psi_\beta(t):=\beta t+(\alpha-\beta)t^3.
\end{equation}
Observe that \( \psi_\beta'(t) = \beta-3(\beta-\alpha)t^2\), and
\[\psi_\beta'(t)=
\begin{cases}
\beta+3(\alpha-\beta)t^2\ge\beta>0,
&0<\beta\le\alpha,\\[2mm]
\beta-3(\beta-\alpha)t^2\ge3\alpha-2\beta\ge0,
&\alpha<\beta\le3\alpha/2\,,
\end{cases}
\]
which  proves the required monotonicity on \([-1,0]\).
Thus, when \(0<\beta\le3\alpha/2\), \(\psi_\beta\) is nondecreasing on \([-1,0]\), and its minimum is attained at \(t=-1\).
If \(\beta>3\alpha/2\), the unique minimizer of
\eqref{eq:scalar-global} is
\[
 t=-\sqrt{\frac{\beta}{3(\beta-\alpha)}}.
\]
Equality in the Rayleigh--Ritz bound requires \(y\in E_\beta\), and as
\(\|y\|^2=1-t^2\), \eqref{eq:axial-minimizers} holds.

\medskip
\noindent
$\bullet$ Step 6: benign character.
By Steps~3--5, every nonpolar local minimum is global.  The south pole is
also global whenever it is a local minimum.  Thus the north pole $p_+$ is the only
possible spurious local minimizer. 
By Steps~1 and~2, $p_+$ is a local minimizer exactly when
\( \lmin(B)>\frac{3\alpha}{2}\).
It is then necessarily spurious because \(f(1,0)=\alpha>-\alpha=f(-1,0)\).
Consequently, no spurious local minimum exists exactly when
\( \lmin(B)\le\frac{3\alpha}{2}\), which proves \eqref{eq:axial-threshold}.
\end{proof}


\begin{thebibliography}{99}

\bibitem{AhmadiEtAl2013}
A.~A. Ahmadi, A.~Olshevsky, P.~A. Parrilo, and J.~N. Tsitsiklis,
\newblock NP-hardness of deciding convexity of quartic polynomials and related
problems,
\newblock \emph{Math. Program.} \textbf{137} (2013), no.~1--2, 453--476.
\newblock \url{https://doi.org/10.1007/s10107-011-0499-2}.



\bibitem{deKlerk2008Survey}
E. de Klerk, 
\newblock The Complexity of Optimizing over a Simplex, Hypercube or Sphere: A Short Survey,
\newblock Central Eur. J. Oper. Res. \textbf{16}, pp. 111--125, 2008.

\bibitem{Faybusovich2004}
L. Faybusovich, 
\newblock Global Optimization of Homogeneous Polynomials on the Simplex and on the Sphere,
\newblock  in \emph{Frontiers in Global Optimization}, 2004, pp. 109--121. 
\bibitem{RongGe2015}
R. Ge, F. Huang, Chi Jin, Yang Yuan, 
\newblock Escaping From Saddle Points -- Online Stochastic Gradient for Tensor Decomposition,
\newblock \emph{COLT 2015}, \emph{PMLR} \textbf{40}, pp. 797--842. 

\bibitem{GeJinZheng2017}
R.~Ge, C.~Jin, and Y.~Zheng,
\newblock No spurious local minima in nonconvex low rank problems: A unified
geometric analysis,
\newblock in \emph{Proceedings of the 34th International Conference on Machine
Learning}, PMLR \textbf{70} (2017), 1233--1242.
\newblock \url{https://proceedings.mlr.press/v70/ge17a.html}.

\bibitem{HillarLim2013}
C.~J. Hillar and L.-H. Lim,
\newblock Most tensor problems are NP-hard,
\newblock \emph{J. ACM} \textbf{60} (2013), no.~6, Article 45.
\newblock \url{https://doi.org/10.1145/2512329}.

\bibitem{Huang2023}
L.~Huang,
\newblock Optimality conditions for homogeneous polynomial optimization on
the unit sphere,
\newblock \emph{Optim. Lett.} \textbf{17} (2023), no.~5, 1263--1270.
\newblock \url{https://doi.org/10.1007/s11590-022-01940-3}.

\bibitem{KoldaMayo2011}
 T. G. Kolda, J. R. Mayo, 
 \newblock Shifted power method for computing tensor eigenpairs, 
 \newblock \emph{SIAM J. Matrix Anal. Appl.} \textbf{ 32}(4),
 pp. 1095--1124, 2011
\bibitem{Lasserrebook}
J.-B.~Lasserre,
\newblock \emph{Moments, {P}ositive {P}olynomials and {T}heir {A}pplications},
\newblock Imperial College Press, 2009.
\bibitem{LasserreOL2022}
J.-B.~Lasserre,
\newblock Homogeneous polynomials and spurious local minima on the unit sphere,
\newblock \emph{Optim. Lett.} \textbf{16} (2022), no.~4, 1105--1118.

\bibitem{LasserreSIAM2022}
J.-B.~Lasserre,
\newblock Optimization on the Euclidean unit sphere,
\newblock \emph{SIAM J. Optim.} \textbf{32} (2022), no.~2, 1430--1445.

\bibitem{LLWY2014}
Zhening Li, Chen Ling, Yiju Wang, Qingzhi Yang, 
\newblock Some advances in tensor analysis and polynomial optimization,
\newblock \emph{Oper. Res. Trans.} \textbf{18}(1), pp. 134--148, 2014

\bibitem{Lim2005}
L.-H. Lim, 
\newblock Singular values and eigenvalues of tensors: A variational approach, 
\newblock in Proceedings of the IEEE International 
Workshop on Computational Advances in Multi-Sensor Adaptive
Processing (CAMSAP), 2005, pp. 129--132. 

\bibitem{MotzkinStraus1965}
T.~S. Motzkin and E.~G. Straus,
\newblock Maxima for graphs and a new proof of a theorem of Tur\'an,
\newblock \emph{Canad. J. Math.} \textbf{17} (1965), 533--540.

\bibitem{MurtyKabadi1987}
K.~G. Murty and S.~N. Kabadi,
\newblock Some NP-complete problems in quadratic and nonlinear programming,
\newblock \emph{Math. Program.} \textbf{39} (1987), no.~2, 117--129.

\bibitem{NieWang2014}
J.  Nie,  Li Wang, 
\newblock Semidefinite Relaxations for Best Rank-1 Tensor Approximations,
\newblock \emph{SIAM J. Matrix Anal. Appl.} \textbf{35} (2014), pp. 1155--1179.


\bibitem{Nie2023}
J. Nie,
\newblock \emph{Moment and Polynomial Optimization}
\newblock SIAM, Philadelphia, 2023
\bibitem{Qi2005}
L. Qi, 
\newblock Eigenvalues of a real supersymmetric tensor, 
\newblock  \emph{J. Symbolic Comput.} \textbf{40}(6), pp.
1302--1324, 2005

\bibitem{Schmudgen2017}
K. Schm\"udgen,
\newblock \emph{The Moment Problem},
\newblock Graduate texts in Mathematics,
\newblock Springer Cham, 2017.

\bibitem{Subag2017}
E.~Subag,
\newblock The complexity of spherical $p$-spin models---a second moment
approach,
\newblock \emph{Ann. Probab.} \textbf{45}(5), pp. 3385--3450, 2017.

\bibitem{Tkachev2010}
V.~G. Tkachev,
\newblock A generalization of Cartan's theorem on isoparametric cubics,
\newblock \emph{Proc. Amer. Math. Soc.} \textbf{138} (2010), no.~8,
2889--2895.
\newblock \url{https://doi.org/10.1090/S0002-9939-10-10385-2}.
\end{thebibliography}
\end{document}